\documentclass[11pt]{article}
\usepackage{amsfonts}
\usepackage{graphicx}
\usepackage{amsmath}
\usepackage{amsthm}
\usepackage{amssymb}
\usepackage{amscd}
\usepackage{color}

\usepackage{graphicx}
\usepackage{microtype}
\usepackage{enumitem}
\usepackage{fullpage}
\usepackage{pinlabel}
\usepackage[top=1.2in,bottom=1.6in,left=1.3in,right=1.3in]{geometry}

\theoremstyle{definition}
\newtheorem{theorem}{Theorem}[section]

\newtheorem{corollary}[theorem]{Corollary}
\newtheorem{construction}[theorem]{Construction}
\newtheorem{convention}[theorem]{Convention}
\newtheorem{example}[theorem]{Example}
\newtheorem{fact}[theorem]{Fact}
\newtheorem{nconj}[theorem]{Naive conjecture}
\newtheorem{lemma}[theorem]{Lemma}
\newtheorem{notation}[theorem]{Notation}

\newtheorem{proposition}[theorem]{Proposition}
\newtheorem{remark}[theorem]{Remark}
\newtheorem{question}[theorem]{Question}
\newtheorem{definition}[theorem]{Definition}

\theoremstyle{plain}

\newtheorem*{claim*}{Claim}

\renewcommand{\H}{\mathbb{H}}

\newcommand{\R}{\mathbb{R}}
\newcommand{\Z}{\mathbb{Z}}
\newcommand{\N}{\mathbb{N}}
\newcommand{\Q}{\mathbb{Q}}

\DeclareMathOperator{\id}{id}
\DeclareMathOperator{\Hom}{Hom}

\DeclareMathOperator{\fix}{fix}
\DeclareMathOperator{\PSL}{PSL}
\DeclareMathOperator{\CO}{CO}
\DeclareMathOperator{\LO}{LO}
\DeclareMathOperator{\Homeo}{Homeo}
\DeclareMathOperator{\Diff}{Diff}

\DeclareMathOperator{\stab}{stab}
\DeclareMathOperator{\ord}{ord}

\title{Group orderings, dynamics, and rigidity}
\author{Kathryn Mann and Crist\'obal Rivas\footnote{K. Mann was partially supported by NSF grant DMS-1606254.  
C. Rivas acknowledges the support of CONICYT via FONDECYT 1150691 and via PIA 79130017.}}

\date{}

\begin{document}
\maketitle

\begin{abstract}
Let $G$ be a countable group.  We show there is a topological relationship between the space $\CO(G)$ of circular orders on $G$ and the moduli space of actions of $G$ on the circle; and an analogous relationship for spaces of left orders and actions on the line.  In particular, we give a complete characterization of \emph{isolated} left and circular orders in terms of \emph{strong rigidity} of their induced actions of $G$ on $S^1$ and $\R$.    

As an application of our techniques, we give an explicit construction of infinitely many nonconjugate isolated points in the spaces $\CO(F_{2n})$ of circular orders on free groups, disproving a conjecture from \cite{BS}, and infinitely many nonconjugate isolated points in the space of left orders on the pure braid group $P_3$, answering a question of Navas.  We also give a detailed analysis of circular orders on free groups, characterizing isolated orders.  
 \end{abstract}
 
\section{Introduction}

Let $G$ be a group.  A \emph{left order} on $G$ is a total order invariant under left multiplication, i.e. such that $a<b$ implies $ga<gb$ for all $a,b,g \in G$.  It is well known that a countable group is left-orderable if and only if it embeds into the group of orientation-preserving homeomorphisms of $\R$, and each left order on a group defines a canonical embedding up to conjugacy, called the \emph{dynamical realization}.  
Similarly, a \emph{circular order} on a group $G$ is defined by a \emph{cyclic orientation cocycle} $c: G^3 \to \{\pm1, 0\}$ satisfying certain conditions (see \S \ref{sec basic facts}), and for countable groups this is equivalent to the group embedding into $\Homeo_+(S^1)$.  Analogous to the left order case, each circular order gives a canonical, up to conjugacy, dynamical realization $G \to \Homeo_+(S^1)$.  This correspondence is the starting point for a rich relationship between the algebraic constraints on $G$ imposed by orders, and the dynamical constraints of $G$--actions on $S^1$ or $\R$.   The correspondence has already proved fruitful in many contexts; one good example is the relationship between orderability of fundamental groups of 3-manifolds, and the existence of certain codimension one foliations or laminations as shown in \cite{CD}.  

For fixed $G$, we let $\LO(G)$ denote the set of all left orders on $G$, and $\CO(G)$ the set of circular orders. These spaces have a natural topology; that on $\CO(G)$ comes from its identification with a subset of the infinite product $\{\pm1, 0\}^{G \times G \times G}$, and $\LO(G)$ can be viewed as a further subset of this (see \S \ref{sec basic facts} for details).     While $\CO(G)$ and $\LO(G)$ have previously been studied with the aim of understanding the structures of orders on groups, our aim here is to relate the spaces $\LO(G)$ and $\CO(G)$ to the moduli spaces $\Hom(G, \Homeo_+(\R))$ and $\Hom(G, \Homeo_+(S^1))$ of actions of $G$ on the line or circle.     

In many cases these moduli spaces are very poorly understood.  An important case is when $G$ is the fundamental group of a surface of genus at least 2.  Here $\Hom(G, \Homeo_+(S^1))$ has a topological interpretation (as the space of flat circle bundles over the surface), yet it
remains an open question whether $\Hom(G, \Homeo_+(S^1))$ has finitely or infinitely many connected components.    
Our work here shows that, for any group $G$, the combinatorial object $\CO(G)$ is a viable tool for understanding the space of actions of $G$ on $S^1$.  

In other cases, actions of $G$ on the circle or line are easier to describe than circular or left orders, and thus the dynamics of actions can serve as a means for understanding the topology of $\LO(G)$ and $\CO(G)$.   A good example to keep in mind is $G = F_2$, since 
$\Hom(F_2, \Homeo_+(S^1)) \cong \Homeo_+(S^1) \times \Homeo_+(S^1)$, but the topology of $\LO(F_2)$ and $\CO(F_2)$ are not so obvious.   

For any group $G$, the space $\LO(G)$ is compact, totally disconnected and, when $G$ is countable, also metrizable \cite{Sikora}.  The same result holds by the same argument for $\CO(G)$.  Consequently the most basic question is whether $\LO(G)$ or $\CO(G)$ has any isolated points -- if not, it is homeomorphic to a Cantor set. 
That $\LO(F_2)$ has no isolated points was proved by McCleary \cite{McCleary} (see also \cite{Navas orders}) and generalized recently to $\LO(G)$ where $G$ is a free product of groups in \cite{Rivas 12}.  That $\CO(F_2)$ has no isolated points either was conjectured in \cite{BS}.   Our techniques give a (perhaps surprising) easy disproof of this conjecture using the dynamics of actions of $F_2$ on $S^1$.

\subsection*{Statement of results.}

Given that $\CO(G)$ is totally disconnected and $\Hom(G, \Homeo_+(S^1))$ typically has large connected components, one might expect little correlation between the two spaces.  Our aim is to demonstrate that there is a strong, though somewhat subtle, relationship.  A first step, and key tool is continuity:

\begin{proposition}[Continuity of dynamical realization] \label{cont prop2}
Let $c$ be a circular order on a countable group $G$, and $\rho$ a dynamical realization of $c$.  For any neighborhood of $U$ of $\rho$ in $\Hom(G, \Homeo_+(S^1))$, there  exists a neighborhood $V$ of $c$ in $\CO(G)$ such that each order in $V$ has a dynamical realization in $U$.   
\end{proposition} 

An analogous result holds for left orders and actions on $\R$.  
With this Proposition and several other tools, we give a complete characterization of isolated left and circular orders in terms of the dynamics (namely, rigidity) of their dynamical realization.  

\begin{theorem} \label{characterization thm}
Let $G$ be a countable group.  A circular order on $G$ is isolated if and only if its dynamical realization $\rho$ is \emph{rigid} in the following strong sense:  for every action $\rho'$ sufficiently close to $\rho$ in $\Hom(G, \Homeo_+(S^1))$ there exists a continuous, degree 1 monotone map $h: S^1 \to S^1$ fixing the basepoint $x_0$ of the realization, and such that $\rho(g) \circ h  = h\circ \rho'(g) $ for all $g \in G$.  
\end{theorem}

\noindent The corresponding result for left orders is Theorem \ref{linear version}.  
\medskip

In the course of the proof of  Theorem \ref{characterization thm}, we establish several other facts concerning the relationship between $\Hom(G, \Homeo_+(S^1))$ and $\CO(G)$.     
When combined with standard facts about dynamics of groups acting on the circle, this gives new information about spaces of circular orders.  For example, we prove the following corollary, a special case of which immediately gives a new proof of the main construction (Theorem 4.6) from \cite{BS}.  

\begin{corollary} \label{dense orbit cor}
Suppose $G \subset \Homeo_+(S^1)$ is a countable group acting minimally, and such that some point $x_0$ has trivial stabilizer.  Then the order on $G$ induced by the orbit of $x_0$ is not isolated in $\CO(G)$.   
\end{corollary}
In particular this shows that, if the dynamical realization of a circular order $c$ on $G$ is minimal, then $c$ is not isolated.   (See Theorem \ref{isolated linear part} and following remarks.)

\paragraph{Isolated orders on free groups and braid groups.} 
In Section \ref{F_n section} we undertake a detailed study of the space of circular orders on free groups.  
As an application of Theorem \ref{characterization thm}, we show:  

\begin{theorem} \label{f_2_iso_thm}
The space of circular orders on the free group on $2n$ generators has infinitely many isolated points.  In fact, there are infinitely many distinct \emph{classes} of isolated points under the natural conjugation action of $F_{2n}$ on $\CO(F_{2n})$.
\end{theorem}

This disproves the conjecture about $\CO(F_2)$ of \cite{BS}.  An alternative, counterexample to the $F_2$ conjecture is given in Section \ref{neighborhood sec}, where we also give an explicit singleton neighborhood of the isolated circular order constructed therein.

We also describe explicitly the dynamics of isolated orders on free groups.  To state  this, let $\{a_1, a_2, ..., a_n\}$ be a set of free generators for $F_n$.  

\begin{theorem} \label{domains thm}
Suppose $\rho$ is the dynamical realization of $c \in \CO(F_n)$.  The order $c$ is isolated if and only if there exist 
disjoint domains $D(s) \subset S^1$ for every $s \in \{a_1^{\pm1}, ..., a_n^{\pm1}\}$, each consisting a finite union of closed intervals, such that 
$\rho(s)(S^1 \setminus D(s^{-1})) \subset D(s)$ holds for all $s$, and the basepoint of $\rho$ is in the complement of the domains $D(s)$.  
\end{theorem}
Note that the conclusion of this theorem is exactly the condition in the classical ping-pong lemma.  

These dynamical realizations have a particularly nice description under the additional assumption that the domains $D(s)$ are connected sets:  

\begin{theorem} 
Let $c \in \CO(F_n)$ have dynamical realization $\rho$ that satisfies the hypotheses of Theorem \ref{domains thm}.  If, additionally, the domains $D(s)$ are connected, then $n$ is even, and the dynamical realization of $c$ is topologically conjugate to a representation $F_n \to \PSL(2, \R) \subset \Homeo_+(S^1)$ corresponding to a hyperbolic structure on a genus $n/2$ surface with one boundary component.    Moreover, each such representation  $F_n \to \PSL(2, \R)$ arises as the dynamical realization of an isolated circular order.  
\end{theorem}

We note that no analog of Theorem \ref{f_2_iso_thm} was previously known for any group, even for left orders.   In \cite{Navas orders}, Navas asked \textit{What can be said in general about the set of isolated (left) orders on a group, up to conjugacy?  For instance, is it always finite?}   A corollary of Theorem \ref{f_2_iso_thm} answers this in the negative: 

\begin{corollary}  The pure braid group $P_3 \cong F_2\times\Z$ has infinitely many distinct conjugacy classes of isolated left orders.
\end{corollary}
\noindent This is proved in $\S$ \ref{LO sec}.  
We expect that the existence of isolated points is not unique to free groups and braid groups and there should be many more examples.  
Further questions are raised in Section \ref{questions sec}.  

\bigskip

\noindent \textbf{Acknowledgements.} 
The authors thank Adam Clay and Andres Navas for their feedback and interest in this work, and Shigenori Matsumoto for helpful comments, including pointing out an error in a previous version.

\section{Background material}\label{sec basic facts}
In this section, we recall some standard facts about left and circular orders.  A reader familiar with orders may wish to skip this section, while the less comfortable reader may wish to consult \cite{Calegari}, \cite{Thurston}, or in the case of left orders, \cite{GOD} for further details. 

\begin{definition}  \label{cocycle def}
[Cocycle definition of circular orders] Let $S$ be a set
. We say that
$ c: S^3\to\{\pm1,0\}$ is a circular order on
$S$ if  \begin{enumerate}[label=\roman*)]
\item $c^{-1}(0)=\triangle(S)$, where $\triangle(S):=\{(a_1,a_2,a_3)\in S^3\mid a_i=a_j,\text{ for some } i\neq j\}$,
\item $c$ is a cocycle, that is $c(a_2,a_3,a_4)-c(a_1,a_3,a_4)+c(a_1,a_2,a_4)-c(a_1,a_2,a_3)=0$ for all $a_1,a_2,a_3,a_4\in S$.

\end{enumerate}
A group $G$ is {\em circularly orderable} if it admits a circular order $c$ which is left-invariant in the sense that $c(u,v,w)=1$ implies $c(gu,gv,gw)=1$ for all $g,u,v,w\in G$.   
\end{definition}

In other words, a circular order on $G$ is a homogeneous 2-cocycle in the standard complex for computing the integral Eilenberg-MacLane cohomology of $G$, which takes the values $0$ on degenerate triples, and $\pm{1}$ otherwise.  

This cocycle condition is motivated by the standard \emph{order cocycle} or \emph{orientation cocycle} for points on the circle.  Say an ordered triple $(x, y, z)$ of distinct points in $S^1$ is \emph{positively oriented} if one can read points $x, y, z$ in order around the circle counterclockwise, and negatively oriented otherwise.  Define the order cocycle $\ord: S^1 \times S^1 \times S^1$ by

$$\ord(x,y,z) = \left\{\begin{array}{rl} 1 & \text{if } (x, y, z) \text{ is positively oriented} \\ -1 & \text{if } (x, y, z) \text{ is negatively oriented} 
\\ 0 & \text{if any two of } x, y \text{ and } z \text{ agree}
.\end{array}
\right. $$
It is easy to check that this satisfies the cocycle condition of Definition \ref{cocycle def}, and is invariant under left-multiplication in $S^1$.  In fact, it is invariant under $\Homeo_+(S^1)$ in the sense that $\ord(x,y,z) = \ord(h(x), h(y), h(z))$ for any orientation preserving homeomorphism $h$.

As mentioned in the introduction, the topology on the space $\CO(G)$ is that inherited from the product topology on $\{\pm1, 0\}^{G \times G \times G}$.   A neighborhood basis of a circular order $c$ consists of the sets of the form 
$\{ c' \in \CO(G) : c'(u,v,w) = c(u,v,w) \text{ for all } u, v, w \in S \}$
where $S$ ranges over all finite subsets of $G$.

\paragraph{Left orders as ``degenerate" circular orders.} \label{parr left orders}
Recall that a left order on a group $G$ is a total order $<$ invariant under left multiplication.  Given $(G, <)$, we can produce a circular order on $G$  by defining $c_<(g_1, g_2, g_3)$ to be the sign of the (unique) permutation $\sigma$ of $\{g_1, g_2, g_3\}$ such that $\sigma(g_1) < \sigma(g_2) < \sigma(g_3)$.

Observe that the left order $c_<$ above is a \emph{coboundary}.  Indeed, if $c'(x,y)$ equals $1$ (respectively $-1$ or $0$) when $x<y$ (respectively $y< x$ or $x=y$), then 
$$c_<(g_1, g_2, g_3) = c'(g_2, g_3)-c'(g_1, g_3)+c'(g_1, g_2).$$ 
Conversely, if a circular order $c$ on a group $G$ is the coboundary of a left-invariant function $c':G^2\to\{\pm 1,0\}$, i.e. we have $c(u,v,w)=c'(v,w)-c'(u,v)+c'(u,v)$ for all $u,v,w\in G$, then one can check that $c'$ defines a left order on $G$ by $x \leq y$ if and only if $c'(x,y)\geq 0$. 
Yet another characterization of the circular order $c_<$ obtained from a left order can be found in \cite[Proposition 2.17]{BS}.

In this sense, we can view $\LO(G)$ as a subset of $\CO(G)$, and give it the subset topology.  It is not hard to see that this agrees with the original topology given in \cite{Sikora}, full details of this are written in \cite{BS}.   Because of this, throughout this paper we frequently take circular orders as a starting point, and treat left orders as a special case.

\paragraph{Dynamical realization.}  We now describe a procedure for realizing a circular order on a group $G$ as an order induced from an action of $G$ on $S^1$.   This starts with the following construction.  

\begin{construction}[Order embedding] \label{embedding def}
Let $G$ be a countable group with circular order $c$, and let $\{g_i\}$ be an enumeration of elements of $G$.  Define an embedding $\iota: G \to S^1$ inductively as follows.  Let $\iota(g_1)$ and $\iota(g_2)$ be arbitrary distinct points.  Then, having embedded $g_1, ..., g_{n-1}$, send $g_{n}$ to the midpoint of the unique connected component of  $S^1 \setminus \{\iota(g_1), ... \iota(g_{n-1})\}$ such that 
$$c(g_i, g_j, g_k) = \ord(\iota(g_i), \iota(g_j), \iota(g_k))$$
holds for all $i, j, k \leq n$.  
\end{construction}

Given an embedding $\iota:G\to S^1$, as above, the left-action of $G$ on itself now gives a continuous, order preserving action by homeomorphism on $\iota(G) \subset S^1$, which extends to a homeomorphism of the closure of $\iota(G)$ in $S^1$. This can be extended to an action by homeomorphisms of $S^1$, for instance by permuting the complimentary interval to the closure of the image by linear maps.  See \cite{Calegari, GOD, Ghys Ens} for more details. We denote the homomorphic embedding just constructed by $\phi_{\iota}:G\to \Homeo_+(S^1)$.

\begin{remark} \label{conj rem}
One can check that taking a different enumeration of the elements of $G$, or a different choice of $\iota(g_1)$ and $\iota(g_2)$ gives embeddings that are conjugate by an orientation-preserving homeomorphism of $S^1$.  The same is true for the corresponding homomorphic embedding $\phi_{\iota}$. Thus it is natural to make the following definition. 
\end{remark}

\begin{definition}[Dynamical realization] \label{dynam def}
Let $G$ be a countable group with circular order $c$.  A \emph{dynamical realization} of $c$ is an embedding $\phi:G \to \Homeo_+(S^1)$ defined as $\phi=h\circ \phi_\iota \circ h^{-1}$, where $\phi_\iota$ is an embedding constructed as above and $h$ is any homeomorphism of $S^1$.  

The \emph{basepoint} of the dynamical realization  $\phi$ is the point $h(\iota(\id))$, where $\id$ denotes the identity element of $G$.  Observe that $\phi(g)(h(\iota(\id)))=h(\iota(g))$.
\end{definition}

By definition, dynamical realizations are all conjugate.  Thus when discussing properties invariant under topological conjugacy, we often speak of {\em the} dynamical realization.

Note that a dynamical realization of a circular order always produces an action of $G$ such that any point in $\iota(G)$ has trivial stabilizer.  Conversely, if $G \subset \Homeo_+(S^1)$ is such that some point $x_0$ has trivial stabilizer, then
$$c(g_1, g_2, g_3) := \ord(g_1(x_0), g_2(x_0), g_3(x_0))$$ 
defines a circular order on $G$.  This is simply the pullback of the order cocycle on $S^1$ under the embedding of $G$ via the orbit of $x_0$.    We say that this is the order \emph{induced by the orbit of} $x_0$. In general, if $G$ does not admit a free orbit, then the circular order on $G/Stab_G(x_0)$ can be extended to a circular order on $G$, for instance, by means of Lemma \ref{completion lemma} below.

For a countable left-ordered group $(G, <)$, there is an analogous construction of a dynamical realization $G \to \Homeo_+(\R)$.  One first enumerates $G$, then defines an embedding $\iota: G \to \R$ inductively by 
$$\iota(g_n) = \left\{ \begin{array}{rcl}  
\max\{ \iota(g_k) : k<n\} + 1 &\text{ if } &g_n > g_k  \text{ for all } k<n \\
\min\{ \iota(g_k) : k<n\} - 1  &\text{ if } &g_n < g_k  \text{ for all } k<n \\
\text{the midpoint of } [g_i, g_j]  &\text{ if } &g_i < g_n < g_j \text{ are successive among } i, j <n
\end{array} \right.
$$
Alternatively, one can check that the construction given above for circular orders produces an action on $S^1$ with a global fixed point whenever $c = c_<$ is a left order (i.e. degenerate) cocycle.  Identifying $S^1 \setminus \{ \ast \}$ with $\R$ gives the dynamical realization of $(G, <)$.

\paragraph{Conjugation.}
We conclude this section by defining the conjugation action of $G$ on its space of orders.   This is important in the study of $\LO(G)$ and $\CO(G)$; indeed, a standard technique to show that an order is \emph{not} isolated is to approximate it by its conjugates \cite{Navas orders, Rivas-Tessera}.  

\begin{definition}[conjugate orders] \label{conj def}
Let $c$ be a circular order on an arbitrary group $G$, and let $g \in G$.  The \emph{g-conjugate order} is the order $c_g$ defined by 
$c_g(x, y, z) = c(xg, yg, zg)$.   An order is called \emph{conjugate} to $c$ if it is of the form $c_g$ for some $g \in G$.    
\end{definition}
Since orders are assumed to be left-invariant, we may equivalently define $c_g(x, y, z) = c(g^{-1}xg, g^{-1}yg, g^{-1}zg)$.  
This gives an action of $G$ on $\CO(G)$ by conjugation, and it is easy to check that this is an action by \emph{homeomorphisms} of $\CO(G)$.  Note that a conjugate of a left order is also a left order, so conjugation also gives an action of $G$ on $\LO(G)$ by homeomorphisms.  
It follows directly from the definition that, given a dynamical realization of $c$ with basepoint $x_0$, a dynamical realization of $c_g$ is given by the same action of $G$ on $S^1$, but with basepoint $g(x_0)$.

\section{A dynamical portrait of left and circular orders} \label{spaces sec} 

We turn now to our main goal of studying the relationship between spaces of actions and spaces of orders.  

Let $\Hom(G, \Homeo_+(S^1))$ denote the set of actions of a group $G$ on $S^1$ by orientation-preserving homeomorphisms, i.e. the set of group homomorphisms $G \to \Homeo_+(S^1)$.  This space has a natural topology; a neighborhood basis of an action $\rho_0$ is given by the sets 
$$O_{(F, \epsilon)}(\rho_0) := \{\rho \in \Hom(G, \Homeo_+(S^1)) \mid  d(\rho(g)(x), \rho_0(g)(x)) < \epsilon \text{ for all } x \in S^1, g \in F\}$$
 where $F$ ranges over all finite subsets of $G$, and $d$ is the standard length metric on $S^1$.  
If $G$ is finitely generated, fixing a generating set $S$ gives an identification of $\Hom(G, \Homeo_+(S^1))$ with a subset of $\Homeo_+(S^1)^{|S|}$ via the images of the generators, and  the subset topology on $\Hom(G, \Homeo_+(S^1))$ agrees with the topology defined above.  

Fixing some point $p \in S^1$, the space $\Hom(G, \Homeo_+(\R))$ of actions of $G$ on $\R$ can be identified with the closed subset 
$$\{ \rho \in \Hom(G, \Homeo_+(S^1) : \rho(g)(p) = p \text{ for all } g \}$$
and its usual (compact--open) topology is just the subset topology.    Given this, our primary focus will be on the larger space $\Hom(G, \Homeo_+(S^1))$ and its relationship with $\CO(G)$; as the $\Hom(G, \Homeo_+(\R)) \leftrightarrow \LO(G)$ relationship essentially follows by restricting to subsets.

As mentioned in the introduction, due to the relationship between circular orders on $G$ and actions of $G$ on $S^1$ given by dynamical realization, it is natural to ask about the relationship between the two spaces $\CO(G)$ and $\Hom(G, \Homeo_+(S^1))$, hoping that the study of one may inform the other.   For instance, one might (naively) propose the following.

\begin{nconj}  \label{naive} 
Let $G$ be a countable group. The space $\CO(G)$ has no isolated points if (or perhaps if and only if) $\Hom(G, \Homeo_+(S^1))$ is connected.   
\end{nconj} 

A supportive example is the case $G = \Z^2$.  It is not difficult to show both that $\Hom(\Z^2, \Homeo_+(S^1))$  is connected and that $\CO(\Z^2)$ has no isolated points.  This kind of reasoning may have motivated the conjecture that $\CO(F_2)$ has no isolated points, since $\Hom(F_2, \Homeo_+(S^1))$ is connected.    However, our disproof of this conjecture for $F_2$ (Theorem \ref{f_2_iso_thm}) shows that the naive reasoning is false.   

Since dynamical realizations are faithful, one might try to improve the Conjecture \ref{naive} by restricting to the subspace of \emph{faithful} actions of a group on $S^1$.  However, the subset of faithful representations in $\Hom(F_2, \Homeo_+(S^1))$ is also connected!  To see this, one can first show that the subset of faithful actions in $\Hom(F_2, \Diff_+(S^1))$ is connected, open, and dense in $\Hom(F_2, \Diff_+(S^1))$ using a transversality argument, as remarked in \cite{Navas 14}.  Since $\Hom(F_2, \Diff_+(S^1))$ is dense in $\Hom(F_2, \Homeo_+(S^1))$, this implies that any faithful action by homeomorphisms can be approximated by one by diffeomorphisms, and hence the space of faithful actions in $\Hom(F_2, \Homeo_+(S^1))$ is also connected.  

As these examples show, the relationship between $\Hom(G, \Homeo_+(S^1))$ and $\CO(G)$ is actually rather subtle. Our next goal is to clarify this relationship. 
  
\begin{convention} 
For the remainder of this paper $G$ will always denote a countable group.  
\end{convention}

\subsection{The relationship between $\Hom(G, \Homeo_+(S^1))$ and $\CO(G)$}

This section provides the groundwork for our dynamical characterization of isolated circular orders, starting with some basic observations.   Let $\Hom(G, \Homeo_+(S^1))/\sim$ denote the quotient of $\Hom(G, \Homeo_+(S^1))$ by the equivalence relation of conjugacy, and equipped with the quotient topology.  Recall that the dynamical realization of a circular order on a countable group is well-defined up to conjugacy in $\Homeo_+(S^1)$.  This defines a natural ``realization" map $R: \CO(G) \to \Hom(G, \Homeo_+(S^1))/\sim$.    Our first proposition is a weaker form of Proposition \ref{cont prop2} from the introduction.  

\begin{proposition}[Dynamical realization is continuous] \label{continuity prop} 
The realization map $R: \CO(G) \to \Hom(G, \Homeo_+(S^1))/\sim$ is continuous.  
\end{proposition}

\begin{proof}
Let $c \in \CO(G)$, and let $\rho$ be a dynamical realization of $c$.  Given a neighborhood $O_{(F, \epsilon)}$ of $\rho$ in $\Hom(G, \Homeo_+(S^1))$, we need to produce a neighborhood $U$ 
of $c$ in $\CO(G)$ such that every circular order $c' \in U$ has a conjugacy representative of its dynamical realization in the $O_{(F, \epsilon)}$--neighborhood of a conjugate of $\rho$.   

Let $S \subset G$ be a finite symmetric set with $F \subset S$ and $|S| > 1/\epsilon$.   After conjugacy of $\rho$, we may assume that $\rho(S)(x_0)$ partitions $S^1$ into intervals of equal length, each of length less than $\epsilon$.   We now show that every circular order $c'$ that agrees with $c$ on the finite set $S \cdot S$ has a conjugate of a dynamical realization in the $O_{(S, \epsilon)} \subset O_{(F, \epsilon)}$-- neighborhood of $\rho$.  

Given such a circular order $c'$, let $\rho'$ be a dynamical realization of $c'$ such that $\rho(g)(x_0) = \rho'(g)(x_0)$ for all $g \in S \cdot S$.    Let $s \in S$.  By construction $\rho(s)$ and $\rho'(s)$ agree on every point of $\rho(S)(x_0)$.   We now compare $\rho(s)$ and $\rho'(s)$ at other points.  Let $I$ be any connected component of $S^1 \setminus S(x_0)$.  Note that $\rho(s)(I) = \rho'(s)(I)$.   Let $y \in I$.   If $\rho(s)(I)$ has length at most $\epsilon$, then as $\rho'(s)(y)$ and $\rho(s)(y)$ both lie in $\rho(s)(I)$, they differ by a distance of at most $\epsilon$.   If $\rho(s)(I)$ has length greater than $\epsilon$, consider instead the partition of $\rho(s)(I)$ by $\rho(S)(x_0) \cap \rho(s)(I)$, this is a partition into intervals of length less than $\epsilon$.    As $s^{-1} \in S$ and $\rho(s^{-1})$ and $\rho'(s^{-1})$ agree on $\rho(S)(x_0)$, considering preimages shows that $\rho(s)(y)$ and $\rho'(s)(y)$ must lie in the same subinterval of the partition, and hence differ by distance at most $\epsilon$. 
\end{proof}

\begin{remark}
Note that the same argument shows that dynamical realization is continuous as a map from $\LO(G)$ to the quotient of $\Hom(G, \Homeo_+(\R))$ by conjugacy in $\Homeo_+(\R)$.  
\end{remark}

The next propositions discuss the partial ``inverse" to the dynamical realization map obtained by fixing a basepoint.    

\begin{notation} Let $G$ be a countable group, and $x_0 \in S^1$.  Let $H(x_0) \subset \Hom(G, \Homeo_+(S^1))$ denote the subset of homomorphisms $G \to \Homeo_+(S^1)$ such that $x_0$ has trivial stabilizer. 
\end{notation}
Each $\rho \in H(x_0)$ induces a circular order on $G$ using the orbit of $x_0$.  This gives a well-defined ``orbit map" $o: H(x_0) \to \CO(G)$.

\begin{proposition}  \label{orbit map prop}
The orbit map $o: H(x_0) \to \CO(G)$ is continuous and surjective.  
\end{proposition} 

\begin{proof}
To show continuity, given a finite set $S \subset G$ and $\rho \in H(x_0)$, we need to find a neighborhood $U$ of $\rho$ in $H(x_0)$ such that the cyclic order of $\rho'(S)(x_0)$ agrees with that of $\rho(S)(x_0)$ for all $\rho' \in U$.  But the existence of such a neighborhood follows immediately from the definition of the topology on $\Hom(G, \Homeo_+(S^1))$.   Surjectivity of the orbit map follows from the existence of a dynamical realization with basepoint $x_0$, as given in Definition \ref{dynam def} and following remarks).  
\end{proof}

To describe the fibers of the orbit map, we will generalize the following statement about left orders from \cite[Lemma 2.8]{Navas orders}.

\begin{lemma}[Navas, \cite{Navas orders}] \label{navas lemma}
Let $<$ be a left order on $G$, and $\rho_1$ its dynamical realization with basepoint $0 \in \R$. Let $\rho_2 \in \Hom(G, \Homeo_+(\R))$ be an action with no global fixed point, and such that $0$ has trivial stabilizer.  The order induced by the $\rho_2$--orbit of $0$ agrees with $<$ if and only if 
there is a non-decreasing, surjective map $f: \R \to \R$, with $f(0) = 0$, such that $\rho_1(g)  f = f  \rho_2(g)$ for all $g \in G$.  
\end{lemma}
The assumption that $f(0) = 0$ is omitted from the statement in \cite{Navas orders}, but it is necessary and used in the proof.  

For circular orders, we replace the non-decreasing, surjective map $f$ above with a continuous \emph{degree 1 monotone map} of $S^1$. This is defined as follows:  A continuous map $f: S^1 \to S^1$ is \emph{degree 1 monotone} if it is surjective and \emph{weakly order-preserving}, meaning that for all triples $x, y, z$, we have $\ord(f(x),f(y),f(z)) = \ord(x, y, z)$ whenever $f(x), f(y)$ and $f(z)$ are distinct points. 
One can check that this condition is equivalent to the existence of a non-decreasing, surjective map $\tilde{f}: \R \to \R$ that commutes with the translation $x \mapsto x+1$, and descends to the map $f$ on the quotient $\R/\Z = S^1 \to S^1$.   

\begin{proposition} \label{semiconj prop}
Let $c_1$ be a circular order on $G$, and $\rho_1$ its dynamical realization with basepoint $x_0$. Let $\rho_2 \in \Hom(G, \Homeo_+(S^1))$ be such that $x_0$ has trivial stabilizer.  The circular order induced by the $\rho_2$ orbit of $x_0$ agrees with $c_1$ if and only if 
there is a continuous, degree 1 monotone map $f: S^1 \to S^1$ such that $f(x_0) = x_0$ and $\rho_1(g) \circ f = f \circ \rho_2(g)$ for all $g \in G$.  
\end{proposition}

The degree 1 monotone map $f$ is an example of a \emph{semi-conjugacy} of $S^1$.  Proposition \ref{semiconj prop} says that the elements of $H(x_0)$ corresponding to the same circular order all differ by such a semi-conjugacy.   However, it is important to note that the relationship between $\rho_1$ and $\rho_2$ given in the proposition is not symmetric.   Loosely speaking, $\rho_1$ (the dynamical realization) can be obtained from $\rho_2$ by collapsing some intervals to points, but not vice-versa.  
This gives a characterization of dynamical realizations as the ``most minimal" (i.e. those with ``densest orbits") actions among all actions with a given cyclic structure on an orbit.

\begin{proof}[Proof of Proposition \ref{semiconj prop}]
Let $\rho_1$ be a dynamical realization of $c$ with basepoint $x_0$, and let $\rho_2$ be an action of $G$ on $S^1$. 
Suppose first that $f$ is a degree 1 monotone map fixing $x_0$, and such that $\rho_1(g) \circ f = f \circ \rho_2(g)$ for all $g \in G$.  
Then the cyclic order of the orbits of $x_0$ under $\rho_1(G)$ and $\rho_2(G)$ agree.  (Here we do not even need $f$ to be continuous.)

For the converse, suppose that $\rho_2$ defines the same circular order as $\rho_1$.  Define a map $f: \rho_2(G)(x_0) \to \rho_1(G)(x_0)$ by 
$\rho_2(g)(x_0) \mapsto \rho_1(g)(x_0)$.   Following the strategy of \cite[Lemma 2.8]{Navas orders}, we show that $f$ first extends continuously to the closure of the orbit $\rho_2(G)(x_0)$, and can then be further extended to a continuous degree 1 monotone map.    Suppose that $x$ is in the closure of the orbit of $x_0$ under $\rho_2(G)$.  Then there exist $g_i$ in $G$ such that $\rho_2(g_i)(x_0) \to x$; moreover we can choose these such that, for each $i$, the triples $\rho_2(g_1)(x_0), \rho_2(g_{i})(x_0), \rho_2(g_{i+1})(x_0)$ all have the same (positive or negative) orientation.  Assume for concreteness that these triples are positively oriented.  

Since the orbit of $x_0$ under $\rho_1$ has the same cyclic order as that of $\rho_2$, the triples $\rho_1(g_1)(x_0), \rho_1(g_{i})(x_0), \rho_1(g_{i+1})(x_0)$ are also all positively oriented, so the sequence $\rho_1(g_i)(x_0)$ is \emph{monotone} (increasing) in the closed interval obtained by cutting $S^1$ at $\rho_1(g_1)(x_0)$.  Thus the sequence $\rho_1(g_i)(x_0)$ converges to some point, say $y$.  Define $f(x) = y$.  This is well defined, since if $\rho_2(h_i)(x_0)$ is another sequence converging to $x$ from the same side, we can find subsequences $g_k$ and $h_k$ such that the triples $\rho_1(g_1)(x_0), \rho_1(h_{k})(x_0), \rho_1(g_{k+1})(x_0)$ are positively oriented, so the monotone sequences $\rho_1(g_k)(x_0)$ and $\rho_1(g_k)(x_0)$ converge to the same point.   
If instead $\rho_2(h_i)(x_0)$ approaches from the opposite side, i.e. the triples $\rho_2(h_1)(x_0), \rho_2(h_{i})(x_0), \rho_2(h_{i+1})(x_0)$ are negatively oriented, then the midpoint construction from the definition of dynamical realization implies that the distance between $\rho_1(g_i)(x_0)$ and $\rho_1(h_i)(x_0)$ approaches zero, so they converge to the same point.   With this definition, the function $f$ is now weakly order preserving on triples, whenever it is defined.  

If $\rho_2(G)(x_0)$ is dense in $S^1$, this completes the definition of $f$.  If not, for each interval $I$ that is a connected component of the complement of the closure of $\rho_2(G)(x_0)$, extend $f$ over $I$ by defining it to be the unique affine map from $I = (a,b)$ to the interval $(f(a), f(b))$ in the complement of the closure of $\rho_1(G)(x_0)$.  This gives a well defined continuous extension that preserves the relation $\rho_1(g) \circ  f = f  \circ \rho_2(g)$ and preserves the weak order preserving property of $f$ on $\rho_2(G)(x_0)$.  
\end{proof}

We conclude these preliminaries by showing the necessity of fixing the basepoint in Propositions \ref{orbit map prop} and \ref{semiconj prop}.   
Note that, if $c$ is an order induced from an action of $G$ on $S^1$ with basepoint $x_0$, then the order induced from the basepoint $g(x_0)$ -- which also has trivial stabilizer -- is precisely the conjugate order $c_g$ from Definition \ref{conj def}.    However, one may also change basepoint to a point outside the orbit of $x_0$.  The following proposition gives a general description of orders under change of basepoint, it will be used in the next section.    We use the notation $\stab(x)$ to denote the stabilizer of a point $x$.  

\begin{proposition}[Change of basepoint]  \label{basepoint prop}
Let $G \subset \Homeo_+(S^1)$ be a countable group.  Let $x \in S^1$, and let $\{x_i\}$ be a sequence of points approaching $x$ such that $\stab(x_i)= \{\id\}$ for all $i$.  
\begin{enumerate}[nosep, label=\roman*)]
\item If $x$ has trivial stabilizer, then the circular orders from basepoints $x_i$ approach the order from basepoint $x$.   In particular, if $x_i$ are in the orbit of $x$ under $G$, then the order induced from the orbit of $x$ can be approximated by its conjugates.  
\item If an interval $I \subset S^1$ satisfies $\stab(x) = \{\id\}$ for all $x \in I$, then all choices of basepoint in $I$ give the same circular order.    
\end{enumerate} 
\end{proposition}

\begin{proof}
i)  Assume $x$ has trivial stabilizer.  Let $S \subset G$ be any finite set.  If $x_j$ is sufficiently close to $x$, then the finite set $\{ g(x_j) \mid g \in S\}$ will have the same circular order as that of $\{ g(x) \mid g \in S\}$.   

ii) Assume now that $\stab(x) = \{\id\}$ for all $x$ in a connected interval $I$.  By part i), under this assumption the map $I \to \CO(G)$ given by sending a point $x$ to the circular order induced from $x$ is continuous.  Since $I$ is connected and $\CO(G)$ totally disconnected, this map is constant.  
\end{proof} 

\begin{remark} \label{basepoint rk}
Note that if $x$ has nontrivial stabilizer, then the set of accumulation points of circular orders induced from such a sequence of basepoints need not be a singleton.  For an easy example, consider an action of $\Z$ on the circle such that some point $x$ is a repelling fixed point of the generator $f$ of $\Z$.  Then for any points $y$ and $z$ close to $x$ and separated by $x$, we will have a positive cyclic order on one and only one of the triples $(y, f(y), f^2(y))$ and $(z, f(z), f^2(z))$.   

Although such an action will not arise as the dynamical realization of any order on $\Z$, this behavior does occur for $\Z$-subgroups of the dynamical realization of orders on many groups, including $F_2$.  
\end{remark}

\subsection{Proof of Theorem \ref{characterization thm} }
We now prove the characterization theorem that was stated in the introduction:  

\vspace{0,3cm}

\noindent \textbf{Theorem \ref{characterization thm}.}
Let $G$ be a countable group.  A circular order on $G$ is isolated if and only if its dynamical realization $\rho$ is \emph{rigid} in the following strong sense:  for every action $\rho'$ sufficiently close to $\rho$ in $\Hom(G, \Homeo_+(S^1))$ there exists a continuous,  degree 1 monotone map $h: S^1 \to S^1$ fixing the basepoint $x_0$ and such that $\rho(g) \circ h  = h\circ \rho'(g)$ for all $g \in G$.  
\bigskip

In particular, this implies that the dynamical realization of an isolated circular order is rigid in the more standard sense that it has a neighborhood in $\Hom(G, \Homeo_+(S^1))$ consisting of a single semi-conjugacy class (see e.g. Definition 3.5 in \cite{Mann survey}).  We remark, however, that this weaker form of rigidity does not entail that the ordering is isolated. For instance, this weak rigidity holds for some non-isolated circular orders on fundamental groups of hyperbolic closed surfaces and also on (some) solvable groups, see Example \ref{ex surfaces} and Example \ref{ex solvable} below.

By minor modifications of the proof, we obtain the same result for left orders. 

\begin{theorem} \label{linear version}
Let $G$ be a countable group.  A left order on $G$ is isolated if and only if its dynamical realization $\rho$ is \emph{rigid} in the following strong sense:  for every action $\rho'$ sufficiently close to $\rho$ in $\Hom(G, \Homeo_+(\R))$ there exists a continuous, surjective monotone map $h: \R \to \R$ fixing the basepoint $x_0$ and such that $\rho(g) \circ h  = h\circ \rho'(g)$ for all $g \in G$. 
\end{theorem}

We prove Theorem \ref{characterization thm} first, then give the modifications for the left order case.  The first step in the proof is a stronger version of the \emph{continuity of dynamical realization} given in Proposition \ref{continuity prop}.

\begin{lemma}[Continuity of dynamical realization, II] \label{lem continuity} 
Let $c \in \CO(G)$ and let $\rho: G \to \Homeo_+(S^1)$ be a dynamical realization of $c$ based at $x_0$. Let $U$ be any neighborhood of $\rho$ in $\Hom(G, \Homeo_+(S^1))$. Then, there exists a neighborhood $V$ of $c$ in $\CO(G)$ such that each order in $V$ has a dynamical realization based at $x_0$ contained in $U$.   
\end{lemma}

\begin{proof} We run a modification of the original argument in Proposition \ref{continuity prop}, which proved a weaker result about the \emph{conjugacy} class of $\rho$. Fix an enumeration $\nu:\N\to G$ such that $\nu(0)=\id$, and let $\rho:G\to \Homeo_+(S^1)$ be the dynamical realization of $c$ build with it. Let $U$ be a neighborhood of $\rho$ in $\Hom(G, \Homeo_+(S^1))$.  We may assume that $U = O_{(F, \epsilon)}(\rho)$ for some $\epsilon$. Suppose first, that $\rho(G)(x_0)$ is dense in $S^1$. Take $S\subset G$ a finite set such that 
\begin{equation} \label{eq dense}\; \rho(g)\rho(S)(x_0)\text{ is }\epsilon/2\text{-dense in $S^1$ for any $g\in F$,}\end{equation} and let $n$ be such that $T=\nu(\{0,\ldots,n\})$ contains $F\cdot S$.  In this way, if $c'$  is a circular order of $G$ coinciding with $c$ over $T$, we can build a dynamical realization $\rho'$ of $c'$ using enumeration $\nu$ that has the same base point as $\rho$, and moreover, has $\rho'(g)=\rho(g)$ on the set $\rho(S)(x_0)$. Further, by (\ref{eq dense}) above, $\rho'$ belongs to  the $O_{(F, \epsilon)}$-neighborhood of $\rho$.

In the case where $\rho(G)(x_0)$ is not dense, we claim that $x_0$ (and hence all points in the orbit of $x_0$) are isolated points.  To see this, suppose that $x_0$ were not isolated.  Then every point of $\rho(G)(x_0)$ would be an accumulation point of the orbit $\rho(G)(x_0)$, and hence the closure of the orbit is a Cantor set.   In this case the ``minimality" of dynamical realizations given by Proposition \ref{semiconj prop}  
implies that $\rho(G)(x_0)$ is in fact dense.  More concretely, if there is some interval $I$ in the complement of the closure of $\rho(G)x_0$, then collapsing each interval $\rho(g)(I)$ to a point produces a new action of $G$ on $S^1$ that is not semi-conjugate to $\rho$ by any degree one \emph{continuous} map $f$, contradicting Proposition \ref{semiconj prop}.   (A more detailed version of this kind of argument is given in Lemma \ref{min set condition} and Corollary \ref{min set 2}.)

Given that $x_0$ is isolated, let $t \in G$ be such that the oriented interval $I:=(x_0, \rho(t)(x_0))$ contains no other points from $\rho(G)(x_0)$.  
Note that, for each pair of distinct elements $g, h \in G$, we have $\rho(g)(I) \cap \rho(h)(I) = \emptyset$.  This is because  
$$ \rho(g)(I) \cap \rho(h)(I) \neq \emptyset \, \,  \Leftrightarrow  I \cap \rho(g^{-1}h)(I) \neq \emptyset $$
and since by definition $I$ contains no points in the orbit of $x_0$, we must have $\rho(g^{-1}h)(I) \supset I$, hence  $\rho(g^{-1}h)^{-1}(I) \subset I$, contradicting the definition of $I$. 

We use this observation to modify the construction from the proof of Proposition \ref{orbit map prop} as follows.  
Given $\epsilon$ and a finite set $F \subset G$, let $S \subset G$ be a finite, symmetric set containing $F\cup\{\id\}$ and such that each interval $J$ in the complement of the set 
$$\bigcup \limits_{g \in S \cdot S} g(I)$$
has length less than $\epsilon$.  
Let $c'$ be a circular order that agrees with $c$ on $S \cdot S$. Then, as in the case where $\rho(G)(x_0)$ is dense, we may build $\rho'$ a dynamical realization of $c'$ that agrees with $c$ on $S \cdot S(x_0)$.   In particular, this means that $\rho(g)(I) = \rho'(g)(I)$ for all $g \in S \cdot S$ and that these intervals are pairwise disjoint.  

For each $g \in S \cdot S$, let $h_g$ be the restriction of $\rho'(g) \rho(g)^{-1}$ to $\rho(g)(I)$.  Note that this is a homeomorphism of $\rho(g)(I)$ fixing each endpoint.   Let $h: S^1 \to S^1$ be the homeomorphism defined by 
$$
h(x) = \left\{\begin{array}{cl} h_g(x) & \text{ if } x \in \rho(g)(I) \text{ for some } g \in S \cdot S \\
x &  \text{otherwise}  
\end{array} \right. 
$$
Then $h \rho' h^{-1}$ is also a dynamical realization of $c'$ (since it is the conjugate of a dynamical realization by a homeomorphismn of $S^1$). We now show that $h \rho' h^{-1}$ is in $O_{(S, \epsilon)}(\rho)$.

Let $s \in S$.   By construction, $\rho(s)$ and $h \rho'(s) h^{-1}$ agree on $S(x_0)$.  Moreover, if $x$ lies in some interval $\rho(t)(I)$ where $t \in S$, we have 
$$h \rho'(s) h^{-1}(x) = 
\rho(st) \rho'(st)^{-1} \rho'(s) \rho'(t) \rho(t)^{-1}(x) = \rho(s)(x)$$ 
so again $\rho(s)$ and $h \rho'(s) h^{-1}$ agree here.  
Finally, if $x$ is not in any such interval, then $\rho(s)(x)$, (and hence $h \rho'(s) h^{-1}(x)$ also) lies in the complement of the set
$\bigcup \limits_{g \in S \cdot S} g(I)$.  Since both images $\rho(s)(x)$ and $h \rho'(s) h^{-1}(x)$ lie in the same complementary interval, which by construction has length less than $\epsilon$, they differ by a distance less than $\epsilon$.  
\end{proof}

We are now in position to finish the proof of the Theorem.  

\begin{proof}[End of proof of Theorem \ref{characterization thm}]
Let $c \in \CO(G)$ be isolated, and let $\rho$ be its dynamical realization with basepoint $x_0$.  Since $c$ is isolated, there exists a finite set $F \subset G$ such that any order that agrees with $c$ on $F$ is equal to $c$.  Because the orbit of $x_0$ under $\rho(F)(x_0)$ is a finite set, there exists a neighborhood $U$ of $\rho$ in $\Hom(G, \Homeo_+(S^1))$ such that, for any $\rho' \in U$, the cyclic order of the set $\rho'(F)(x_0)$ agrees with that of $\rho(F)(x_0)$.   Let $\rho' \in U$.  If $x_0$ has trivial stabilizer under $\rho'$, then the cyclic order on $G$ induced by the orbit of $x_0$ under $\rho'(G)$ agrees with $c$ on $F$, so is equal to $c$. By Proposition \ref{semiconj prop}, this gives the existence of a map $h$ as in the theorem.  

If $x_0$ instead has non-trivial stabilizer, say $K \subset G$, then the orbit of $x_0$ under $\rho'(G)$ gives a circular order on the set of cosets of $K$.  Lemma \ref{completion lemma} below implies that this can be ``extended" or \emph{completed} to an order on $G$ in at least two different ways, both of which agree with $c$ on $F$. In particular, one of the order completions is not equal to $c$.  This shows that $c$ is not an isolated point, contradicting our initial assumption.  This completes the forward direction of the proof.  

For the converse, assume that $c$ is a circular order whose dynamical realization $\rho$ satisfies the rigidity condition in the theorem. Let $U$ be a neighborhood of $\rho$ so that each $\rho'$ in $U$ is semi-conjugate to $\rho$ by a continuous degree one monotone map $h$ as in the statement of the theorem.   Lemma \ref{lem continuity} provides a neighborhood $V$ of $c$ such that any $c'\in V$ has a dynamical realization in $U$.   Proposition \ref{semiconj prop} now implies that the neighborhood $V$ consists of a single circular order, so $c$ is an isolated point.  
\end{proof}

\begin{lemma}[Order completions, see Theorem 2.2.14 in \cite{Calegari}] \label{completion lemma}
Let $G \subset \Homeo(S^1)$ and let $x$ be a point with stabilizer $K \subset G$.   Any left order on $K$ can be extended to circular order on $G$ that agrees with the circular order on cosets of $K$ induced by the orbit of $x$.  
\end{lemma}

We include a proof sketch for completeness. 

\begin{proof}[Proof sketch] Suppose that $<_K$ is a left order on $K$.  Define $c(g_1,g_2,g_3)=1$ whenever $\ord(g_1(x),g_2(x),g_3(x))=1$. When $(g_1,g_2,g_3)$ is a non-degenerate triple but $g_1(x)=g_2(x)\neq g_3(x)$, then one can declare $c(g_1,g_2,g_3)=1 $ whenever $g_2^{-1}g_1 <_K$.  This determines also the other cases where exactly two of the points $g_i(x)$ coincide.   The remaining case is when $g_1(x)=g_2(x)=g_3(x)$, in which case we declare $c(g_1,g_2,g_3)$ to be the sign of the permutation $\sigma$ of $\{\id, g_1^{-1}g_2, g_1^{-1}g_3\}$ such that $\sigma(\id) < \sigma(g_1^{-1}g_2) < \sigma(g_1^{-1}g_3)$.   Checking that this gives a well-defined left-invariant circular order is easy and left to the reader.
\end{proof}

We end this section with the modifications necessary for the left order version of this theorem.  

\begin{proof}[Proof of Theorem \ref{linear version}] 
Since any linear order is in particular a circular order, the argument from Lemma \ref{lem continuity} can also be used to show the following 
\begin{quote} \textit{
Let $G$ be a countable group, and $<$ a left order on $G$, with dynamical realization $\rho$.  Let $U$ be a neighborhood of $\rho$ in $\Hom(G, \Homeo_+(\R))$. Then there exists a neighborhood $V$ of $<$ in $\LO(G)$ such that each order in $V$ has a dynamical realization in $U$.   
}
\end{quote}
Combining this with Lemma \ref{navas lemma} (in place of Proposition \ref{semiconj prop}, which was used in the circular order case) now shows that any order $<$ on $G$ with dynamical realization $\rho$ satisfying the rigidity assumption is an isolated left order.   

For the other direction of the proof, assuming $<$ is an isolated left order, one runs the beginning of the proof of Theorem \ref{characterization thm}:  since $<$ is isolated, there exists a finite set $F \subset G$ such that any order that agrees with $<$ on $F$ is equal to $<$.  Because the orbit of $x_0$ under $\rho(F)(x_0)$ is a finite set, there exists a neighborhood $U$ of $\rho$ in $\Hom(G, \Homeo_+(\R))$ such that, for any $\rho' \in U$, the linear order of the set $\rho'(F)(x_0)$ agrees with that of $\rho(F)(x_0)$.   Let $\rho' \in U$.  Order completion applied here again allows us to assume that $x_0$ has trivial stabilizer under $\rho'$ (applying order completion in this case will give a left rather than circular order), and so the cyclic order on $G$ induced by the orbit of $x_0$ under $\rho'(G)$ agrees with $<$ on $F$, so is equal to $<$.  By Lemma \ref{navas lemma}, this gives the existence of a map $h$ as in the theorem.   
\end{proof}

\subsection{Convex subgroups and dynamics}  \label{more properties sec}

Using our work above, we give some additional properties of isolated circular orders.  
For this, we need to describe certain subgroups associated to a circular order.  Recall that a subgroup $H$ of a left-ordered group is called \emph{convex} if, whenever one has $h_1 < g < h_2$ with $h_1, h_2$ in $H$ and $g \in G$, then $g \in H$.   We extend this to circularly ordered groups as follows.  (A similar definition appears in \cite{Jakubik}.)

\begin{definition}
A subgroup $H$ of a circularly ordered group is \emph{convex} if the restriction of $c$ to $H$ is a left order, and if whenever one has $c(h_1, g, h_2) = +1$ and $c(h_1, \id, h_2) = +1$ with $h_1, h_2$ in $H$ and $g \in G$, then $g \in H$.  
\end{definition}
In particular  $\{\id \}$ is (trivially) a convex subgroup.

\begin{lemma}[The linear part of an action] \label{linear part lem}
Let $c$ be a circular order on $G$.  Then there is a unique maximal convex subgroup $H \subset G$.  
We call $H$ the \emph{linear part} of $c$.  
\end{lemma}

\begin{proof} 
Let $\rho$ be a dynamical realization of a circular order $c$ on $G$, with basepoint $x_0$.  Note that the definition of convex is easily seen to be equivalent to the following condition: \\
$(\ast)$ \textit{ $\rho(H)$ acts on $S^1$ with a fixed point, and if $g(x_0)$ lies in the connected component of $S^1 \setminus \fix(\rho(H))$ containing $x_0$, then $g \in H$.} \\
In other words, if $I_H$ denotes the connected components of $S^1 \setminus \fix(\rho(H))$ containing $x_0$, then $H = \{g \in G : g(x_0) \in I_H \}$.  
Identifying $I_H$ with the line, the induced linear order on $H$ agrees with the order on the orbit of $x_0$ under $\rho(H)$.   

Now, assume that $H$ and $K$ are two proper convex subgroups. Then, since $H\cup K\neq G$, and $g(x_0)\notin I_H\cup I_K$ for $g\in G\setminus(H\cup K)$, we have that $I_H\cup I_K\neq S^1$. Therefore, the cyclic order in $G$ induces a linear order on $H\cup K$ that we denote $<$. Assume $h \in H \setminus K$, then (up to replacing $h$ with its inverse) we have $h(x_0) > k(x_0)$ for all $k \in K$, and $h^{-1}(x_0) < k(x_0)$ for all $k \in K$.  Thus, $H \supset K$.   It follows that the union of convex subgroups is convex, and so the union of all convex subgroups is the (unique) maximal element.  
\end{proof}

To give a more complete dynamical description of the linear part of a circular order, we use the following general fact about groups acting on the circle.  

\begin{lemma}  \label{trichotomy} 
Let $G$ be any group, and $\rho: G \to \Homeo_+(S^1)$ any action on the circle. Then there are three mutually exclusive possibilities.
\begin{enumerate}[nosep, label=\roman*)]
\item There is a finite orbit.  In this case, all finite orbits have the same cardinality. 
\item The action is minimal, i.e. all orbits are dense.
\item There is a (unique) compact $G$--invariant subset $K \subset S^1$, homeomorphic to a Cantor set, and contained in the closure of any orbit.   This set is called the \emph{exceptional minimal set}.  
\end{enumerate} 
\end{lemma}
A proof can be found in \cite[Sect. 5]{Ghys Ens}.   

\begin{proposition}[Dynamical description of the linear part] \label{linear part prop}
Let $\rho$ be the dynamical realization of a circular order $c$.   According to the trichotomy above, the linear part of $c$ has the following description.
\begin{enumerate}[nosep, label=\roman*)]
\item (finite orbit case) the linear part is the finite index subgroup of $G$ that fixes any finite orbit pointwise.   
\item (minimal case) the linear part is trivial.
\item (exceptional case) the linear part is the stabilizer of the connected component of $S^1 \setminus K$ that contains $x_0$.   
\end{enumerate} 
\end{proposition}

\begin{proof}  
We use the alternative characterization $(\ast)$ of the linear part of an order given in the proof of Lemma \ref{linear part lem}.  Suppose first that $\rho(G)$ has a finite orbit, say $O$, and let $H$ be the finite index subgroup fixing this orbit pointwise.  Then $H$ is a left-ordered subgroup since it acts on $S^1$ with a fixed point, and $H$ is convex with $I_H$ equal to the connected component of $S^1 \setminus O$ containing $x_0$.  Moreover, for any $g \notin H$, the action of $\rho(g)$ has no fixed point (it cyclically permutes the connected components of $S^1 \setminus O$ and so $g$ is not in any convex subgroup.  This shows that $H$ is the maximal convex subgroup of $G$.

More generally, if $H$ is a convex subgroup of any circular order on a group $G$, and $g \notin H$, then $\rho(g)(I_H) \cap I_H$ is always empty.  To see this, note that the endpoints of $I_H$ are accumulation points of $\rho(H)(x_0)$, so if $\rho(g)(I_H) \cap I_H \neq \emptyset$, then there exists some element of the form $\rho(gh)(x_0) \in I_H$.  Since $H$ is convex, this means that $g \in H$.  

This argument shows that $I_H$ is always a \emph{wandering interval} for the action of $\rho(G)$, and in particular, that the action of $G$ cannot be minimal.  
Moreover, when the maximal convex subgroup $H$is non-trivial, $I_H$ is the maximal (with respect to inclusion) wandering interval containing $x_0$, namely a connected component of the complement of the exceptional minimal set.  
This completes the proof.  
\end{proof}

We now state and prove the main result in this section.

\begin{theorem} \label{isolated linear part} 
If $c$ is an isolated circular order on an infinite group $G$, then $c$ has nontrivial linear part, and the induced left order on the linear part is an isolated left order.  
\end{theorem} 

In particular, using Proposition \ref{linear part prop}, this implies that the dynamical realization of any isolated circular order on any group $G$ cannot be minimal.  This was Corollary \ref{dense orbit cor} stated in the Introduction, and it  immediately gives a new proof of the main construction (Theorem 4.6) of \cite{BS}.   The fact that the linear part of an isolated circular order has an isolated linear order implies also that it has a nontrivial \emph{Conradian soul}, as defined in \cite{Navas orders}, and that the soul admits only finitely many left orders.  

For the proof, we will use a consequence of the following theorem of Margulis.  

\begin{proposition}[Margulis, \cite{Margulis}] \label{margulis prop} 
Let $G$ be any group, and $\rho: G \to \Homeo_+(S^1)$ any action on the circle.  Either $\rho(G)$ preserves a probability measure on $S^1$, or $\rho(G)$ contains a nonabelian free subgroup.  
\end{proposition}

The construction of the nonabelian subgroup comes from the existence of \emph{contracting intervals},  meaning intervals $I$ such that there exists a sequence $g_n$ in $G$ such that the diameter of $g_n(I)$ approaches $0$.  An exposition of the proof can be found in  \cite{Ghys Ens}.   In the case where the action of $G$ is minimal, one can find contracting intervals containing any point.  Put more precisely, we have the following corollary of the proof given in \cite{Ghys Ens}.   

\begin{proposition} \label{minimal prop}[Ghys \cite{Ghys Ens}]
Let $\rho: G \to \Homeo_+(S^1)$ be a \emph{minimal} action on the circle.  
Either $\rho(G)$ is conjugate to a group of rotations, or the following condition holds: For any point $x\in S^1$, there is a neighborhood $I$ of $x$ such that for any $y\in S^1$, there is a sequence of elements $g_{y,n}\in G$ such that $g_{y,n}(I)$ converges to $y$ as $n$ tends to $\infty$.

\end{proposition} 

We say that the neighborhood $I$ in the above proposition is a {\em contractible} neighborhood of $x$.

In order to prove Theorem \ref{isolated linear part}, we also need to describe dynamical realizations with exceptional minimal set.  We state this as a separate lemma; it will be used again in the next section.  

\begin{lemma}[Condition for dynamical realizations]  \label{min set condition}
Suppose that $\rho: G \to \Homeo_+(S^1)$ is such that $x_0$ has trivial stabilizer, and $\rho$ has an exceptional minimal set $K$.  If $\rho$ is a dynamical realization of a circular order with basepoint $x_0$, then each connected component $I$ of $S^1 \setminus K$ contains a point of the orbit $\rho(G)(x_0)$ and has nontrivial stabilizer in $G$.  
\end{lemma} 

\begin{proof}
Let $\rho: G \to \Homeo_+(S^1)$ be a dynamical realization with basepoint $x_0$ and exceptional minimal set $K$.  
Since $\rho(G)$ permutes the connected components of $S^1 \setminus K$, it suffices to show that each interval of $S^1 \setminus K$ contains at least two points in the orbit $\rho(G)(x_0)$.  

Suppose for contradiction that some connected component $I$ of $S^1 \setminus K$ contains one or no points in the orbit $\rho(G)(x_0)$.  Then, for each $g \in G$, the connected component $\rho(g)(I)$ of $S^1 \setminus K$ also contains at most one point in the orbit of $x_0$.  Collapsing each interval $\rho(g)(I)$ to a point gives a new circle, on which $G$ acts by homeomorphisms.  Let $\rho_2$ denote this new action, and note that $\rho_2(G)(x_0)$ has the same cyclic order as $\rho(G)(x_0)$.     More precisely, let $h: S^1 \to S^1$ be a map such that, for each $g \in G$, the image $h(\rho(g)(I))$ is a singleton, $h$ is injective on the complement of $\bigcup_g \rho(g)(I)$ and $h$ fixes $x_0$.  Define $\rho_2$ by 
$h \circ \rho(g) = \rho_2(g) \circ h$.  

Since $\rho$ is a dynamical realization, Proposition \ref{semiconj prop} gives a continuous degree one monotone $f: S^1 \to S^1$ fixing $x_0$ and such that $f \circ \rho_2(g) = \rho(g) \circ f$.  In other words, $f$ is an inverse for $h$ on the orbit of $x_0$.   However, since the endpoints of $I$ lie in the exceptional minimal set $K$, they are also in the closure of the orbit $\rho(G)(x_0)$.  But $h(I)$ is a point, so $f$ cannot be continuous at this point.  
This gives the desired contradiction.  
\end{proof}

\begin{proof}[Proof of Theorem \ref{isolated linear part}]
Suppose that $c$ is an isolated circular order on $G$.   We use proposition \ref{trichotomy} to describe its dynamical realization $\rho$.  Let $x_0$ be the basepoint of $\rho$.  First, we will show that $\rho$ cannot be minimal.  Assume for contradiction that $\rho$ is minimal.  Then Proposition \ref{minimal prop} implies that either $G$ is abelian and $\rho$ conjugate to an infinite group of rotations, or we have an interval $I$ containing $x_0$ and, for each $y \in S^1$ a sequence contracting $I$ to $\{y\}$.   The rotations case gives an order which is not isolated, this is shown in Lemma \ref{abelian lem} below. (While it is relatively easy to see that such an order on an abelian group is not isolated if, say, $G$ has rank at least one, it is more difficult for groups like the group of rotations by $n \pi/2^k, n, k \in \Z$ and deserves a separate lemma.) 

Thus, we now have only to deal with the second case of a contractible interval.  By the proof of Proposition \ref{basepoint prop}, there is an interval $I_0=[a,b]$ containing $x_0$ in its interior, such that the circular order induced from any point $y\in I$ with trivial stabilizer coincides with $c$. By taking a smaller neighborhood of $x_0$ if necessary, we may assume that $I$ is a {\em contractible} neighborhood of $x_0$. Let $y_1$, $y_2$ and $z$ in $I$ be such that $y_1,z,y_2,x_0$ are in counterclockwise order. By minimality, we may assume that all these points have trivial stabilizer.  Let $(g_{i,n})_n$ ($i=1,2$) be a sequence of elements in $G$ such that $g_{i,n}(I)$ converges to $y_i$. Let $n\geq 0$ be such that 
$$g_{1,n}(I), \; \{z\},\; g_{2,n}(I), \{x_0\}$$ 
are in counterclockwise order. Then, $c(id,g_{1,n},g_{2,n})=1,$ but if $c_z$ is the circular order induce from $z$, then $c_z(id,g_{1,n},g_{2,n})=-1$,  contradicting our assumption on $I$ that any order induce form a point in it must coincide with $c$. Thus the dynamical realization of an isolated circular order cannot be minimal.



Now we describe the case when  $\rho$ is not minimal.  By Lemma \ref{trichotomy}, it has either a finite orbit or invariant Cantor set.  In the first case, since $G$ is infinite, $x_0$ is not an element of a finite orbit, and so $c$ has a nontrivial linear part.    In the second case, Proposition \ref{min set condition} implies that $x_0$ lies in the complement of the exceptional minimal set $K$, and that the stabilizer of the connected component of $S^1 \setminus K$ containing $x_0$ is nontrivial.  Hence, $c$ has nontrivial linear part.  

Now let $H \subset G$ denote the linear part of $c$.  If the left order on $H$ were not isolated in $\LO(H)$, then order completion from Lemma \ref{completion lemma} allows us to approximate $c$ in $\CO(G)$ using left orders on $H$ approaching the restriction of $c$ to $H$.  Thus, the left order on $H$ must be isolated.    
\end{proof}

It remains only to prove the lemma on abelian groups.  

\begin{lemma}[Infinite groups of rotations are not isolated] \label{abelian lem}
Let $G$ be an infinite group, and suppose $\rho(G) \subset \mathrm{SO}(2) \cong \R/\Z$ is the dynamical realization of a circular order $c$.  Then $c$ is not isolated.  
\end{lemma}

\begin{proof}
Identify $G$ with its image in the additive group $\R/\Z$.   Note that the order on $G \subset \R/\Z \cong S^1$ agrees with the usual cyclic order on points of $S^1$.  

\noindent \textit{Case i. $G$ is not a torsion group.}
Let $\hat{G} \subset \R$ be the set of lifts of elements to $\R$, so we have a short exact sequence $0 \to \Z \to \hat{G} \to G \to 0$, and consider the vector space $V$ over $\Q$ generated by $\hat{G} \subset \R$. By assumption, $G$ is not torsion, so $V \not\subset \Q $.  Let $\lambda \in V \setminus \Q$.  Choose $\lambda'\in\R$ linearly independent over $\Q$ from $V$, and define $\phi: V \to \R$ by $\lambda q \mapsto \lambda' q$ for $q \in \Q$, and $\alpha \mapsto \alpha$ for any $\alpha$ in the complement of the span of $\lambda$.  
Then $\phi(\hat{G}) \cong \hat{G}$ as additive groups, and $\phi$ descends to an embedding of $G$ in $\R/\Z$ with a different cyclic order.  This order can be made arbitrarily close to the original one by taking $\lambda'$ as close as we like to $\lambda$.  
\medskip

\noindent \textit{Case ii. $G \subset \Q/\Z$. }
In this case $G$ is an infinite abelian torsion group and we can decompose it into a direct sum of groups $G = \oplus G_p$, where $G_p$ is the group of all elements who have order a power of $p$, for each prime $p$. These are simply the elements $a/p^k \in \R/\Z$.  We use the following basic fact. 

\begin{fact}
Let $A_p = \{a/p^k : k \in \N \} \subset \R/\Z$.   Then, for any $k$, the function $x \mapsto x+ p^k x$ is an automorphism of $A_p$.  
\end{fact} 
To see this, one checks easily that any map
$x \mapsto \sum_{i=0}^\infty a_i p^i x$ 
gives a well defined endomorphism. (For fixed $x \in A_p$, all but finitely many terms in the formal power series $\sum a_i p^i x$ vanish mod $\Z$). Since $1 + p^k$ is invertible in the $p$-adic integers, $x \mapsto x+ p^k x$ has an inverse, so is an automorphism.  

Given any finite subset $S$ of $G$, we can find some $G_p \subset G$ and $N > 0$ such that $G_p \cap \{a/p^k : k>N \}$ is nonempty and does not contain any element of $S$.   Let $k$ be the smallest integer greater than $N$ such that $G_p$ contains an element of the form $a/p^k$.   Using the fact above, define a homomorphism $G \to \R/\Z$ to be the identity on $G_q$ for $q \neq p$, and to be $x \mapsto x+ p^{k-1} x$ on $G_p$.    Note that this is well defined, injective, restricts to the identity on $S$, and changes the cyclic order of elements of $G_p$. 
\end{proof} 

We conclude this section with a converse to Lemma \ref{min set condition} and corollary to the proof of Proposition \ref{semiconj prop}.  It will also be useful in the next section.  

\begin{corollary} \label{min set 2}
Suppose that $\rho: G \to \Homeo_+(S^1)$ is an action with exceptional minimal set $K$, that $\rho(G)$ acts transitively on the set of connected components of $S^1 \setminus K$, and that, for some component $I$, the stabilizer of $I$ is nontrivial and acts on $I$ as the dynamical realization of a linear order with basepoint $x_0 \in I$.  Then $\rho$ is the dynamical realization of a circular order with basepoint $x_0$.
\end{corollary}

\begin{proof} Let $\rho$ be as in the proposition and let $x_0\in I$. Then $x_0$ has trivial stabilizer.  Suppose for contradiction that $\rho$ is \emph{not} a dynamical realization, and let $\rho_2$ be the dynamical realization of the circular order on $G$ given by the orbit of $x_0$ under $\rho(G)$.  Proposition \ref{semiconj prop} then gives the existence of a continuous degree one monotone map $f: S^1 \to S^1$ such that $f(x_0) = x_0$ and $\rho_2(g) \circ f = f \circ \rho(g)$.  If $f$ is a homeomorphism, then $\rho$ is a dynamical realization.  Thus, $f$ must not be injective.  

Given that $f$ is not injective, the construction of $f$ in the proof of Proposition \ref{semiconj prop} implies that there is a connected component $J$ of the complement of the closure of $\rho_2(G)(x_0)$ such that $f(J)$ is a singleton, say $y$.  Since the exceptional minimal set $K$ is contained in the closure of $\rho_2(G)(x_0)$, it follows that $J$ lies inside a connected component of $S^1\setminus K$. After conjugacy, we may find such an interval $J$ such that $J \subset I$. But then, since by assumption $\rho$ restricts to a dynamical realization of the stabilizer $\stab_G(I)$, Lemma \ref{navas lemma} provides a continuous surjective map $\hat f :I\to I$ such that $\rho(g)\circ \hat f = \hat f \circ \rho_2(g)$ for every $g$ fixing $I$. But $\hat f$ is local inverse of $f$, so it cannot be continuous. This provides the desired contradiction.
\end{proof}


\subsection{Illustrative non-examples}

We give some examples to show that Theorem \ref{characterization thm} does not hold when $\rho$ is assumed to have a related or slightly weaker form of rigidity.   

\begin{example}  \label{ex surfaces}
Let $G$ be the fundamental group of a closed surface $\Sigma_g$ of genus $g \geq 2$, and $\rho: G \to \PSL(2,\R) \subset \Homeo_+(S^1)$ the homomorphism arising from a hyperbolic structure on $\Sigma_g$. This is an embedding of $G$ into $\PSL(2,\R)$ as a cocompact lattice.  
Let $x_0 \in S^1$ be a point with trivial stabilizer.  Such a point exists, as there are only countably many points with nontrivial stabilizer, each one an isolated fixed point of an infinite cyclic subgroup of $G$.   Then the orbit of $x_0$ induces a circular order, say $c_0$, on $G$, and since $\rho$ is minimal Proposition \ref{semiconj prop} implies that $\rho$ is a dynamical realization of $c_0$ with basepoint $x_0$.  By Corollary \ref{dense orbit cor}, this is not an isolated circular order.  However, the action does have a form of rigidity, which we now describe.  

The main theorem of Matsumoto \cite{Matsumoto invent}, together with minimality of $\rho$ implies that there exists a neighborhood $U$ of $\rho$ in $\Hom(G, \Homeo_+(S^1))$ such that, for all $\rho'$ in $U$, there exists a continuous, degree one monotone map $h$ such that 
$h \circ \rho' = \rho \circ h$.
This is quite similar to our ``strong rigidity", except that here $h$ will not generally fix the basepoint $x_0$  

To see directly that $c_0$ is not isolated, one can change the basepoint as in Remark \ref{basepoint rk} to approximate $c_0$ by its conjugates.   This corresponds to conjugating $\rho$ by some small homeomorphism $h$ that does not fix $x_0$.  
\end{example}

One can modify the group $G$ in the example above to give a dynamical realization $\rho$ of a circular order $c$ with \emph{nontrivial linear part} (and isolated left order on the linear part!) and so that $\rho$ still has some form of rigidity --  but where $c$ fails again to be an isolated circular order due to basepoint considerations.   Here is a brief sketch of one such construction.   

\begin{example}
Let $G = \pi_1(\Sigma_g) \times \Z$.   Define $\rho: G \to \Homeo_+(S^1)$ by starting with the action of $\pi_1(\Sigma_g)$ defined above, then ``blowing up" each point in the orbit of $x_0$, replacing it with an interval (i.e. performing the Denjoy trick), and inserting an action of $\Z$ by translations supported on these intervals, that commutes with the action of $\pi_1(\Sigma_g)$.    Corollary \ref{min set 2} shows that these are dynamical realizations of the circular order on $G$ obtained from the orbit of $x_0$.  

Moreover, much like in the case above, one can argue from Matsumoto's result that any nearby action $\rho'$ of $G$ on $S^1$ is \emph{semi-conjugate} to $\rho$, in the sense that both can be collapsed to a common minimal action where the $\Z$ factor acts trivially.  However, $\rho$ is not an isolated circular order; one can produce arbitrarily nearby orders by performing the same construction, but blowing up the orbit of a nearby point instead of $x_0$, and choosing the basepoint there. 
\end{example}

Our last example concerns circular orders on solvable groups.

\begin{example}\label{ex solvable}  Consider the Baumslag-Solitar groups $BS(1,2)=\langle a,b\mid aba^{-1}=b^2\rangle$ acting by $\rho(a):x\mapsto 2x$ and $\rho(b):x\mapsto x+1$ for $x\in \R\cup\{\infty\}=S^1$. As in Example \ref{ex surfaces}, $\rho$ is a dynamical realization of any circular order induced from a point $x_0$ having trivial stabilizer.  For this, one can take $x_0$ to be any point in $\R \setminus \Q$. 

We claim that the representation $\rho$ is rigid, in the sense that all nearby actions of $BS(1,2)$ on the circle are semi-conjugate to it.   Indeed, from the invariance under conjugation of rotation number,  $b$ always has a fixed point in $S^1$. Furthermore, since the fixed points of $\rho(a)$ are hyperbolic, and therefore stable, for any representation $\rho'$ close to $\rho$, the element $\rho'(a)$ has a fixed point.  By iterating a fixed point of $\rho'(b)$ under $\rho'(a^{-1})$, we obtain a global fixed point of  $\rho'$. This implies that, up to semi-conjugacy, we can view $\rho'$ as an action of $BS(1,2)$ on $\R \cup\{\infty\}$ fixing $\infty$. Moreover, in this model, $\rho'(a)$ has a  fixed point, say $p$, on $\R$ and we have $\rho'(b)(p) > p$.  But now, in \cite{Rivas jgt} it is shown that $BS(1,2)$ has only four semi-conjugacy classes of actions on the line. Two of them giving actions where $a$ has no fixed points, and the other two are affine actions: one in which $b$ is a negative translation and the other in which $b$ is a positive translation. This implies that $\rho'$ is semi-conjugate to $\rho$. However, the ordering on $BS(1,2)$ induced from $x_0$ is not isolated, it is approximated by its conjugates in the same way as in Example \ref{ex surfaces}.

Similar arguments can be applied to many orderings on (not necessarily affine) solvable groups. See \cite{Rivas-Tessera}. 

\end{example}

\section{Circular orders on free groups} 
\label{F_n section}

In this section, we use the results of Section \ref{spaces sec} to show that there are infinitely many nonconjugate circular orders on free groups of even rank, and characterize the dynamical realizations of isolated circular orders on free groups generally, proving Theorems \ref{f_2_iso_thm} and  \ref{domains thm}. 

We start with a definition related to the conditions in Theorem \ref{domains thm}.    

\begin{definition} 
Let  $a_1, a_2, ... , a_n \in  \Homeo_+(S^1)$.  We say these elements have \emph{ping-pong dynamics} if there exist pairwise disjoint closed sets $D(a_i)$ and $D(a_i^{-1})$ such that, for each $i$, we have 
$$\rho(a_i) \big( S^1 \setminus D(a_i^{-1}) \big) \subset D(a_i).$$
\end{definition}

We call any such sets $D(a_i)$ and $D(a_i^{-1})$ satisfying $a_i \big( S^1 \setminus D(a_i^{-1}) \big) \subset D(a_i)$ \emph{attracting domains} for $\rho(a_i)$ and $\rho(a_i^{-1})$ respectively, and 
use the notation $D(a_i^{\pm 1})$ to denote $D(a_i) \cup D(a_i^{-1})$.    These attracting domains need not be connected.   In this case, we use the notation $D_1(s), D_2(s), ...$ for the connected components of $D(s)$.   Note that the definition of ping-pong dynamics implies that for each $s \in \{a_1^{\pm 1}, ..., a_n^{\pm 1}\}$, and for each domain $D_k(t)$ with $t \neq s^{-1}$, there exists a unique $j$ such that $s(D_k(t))$ lies in the interior of $D_j(s)$.  

The terminology is motivated by the following lemma, a version of the classical ping-pong lemma.  

\begin{lemma}[Ping-pong] \label{lem pingpong}
Let $a_1, a_2, ... , a_n  \in \Homeo_+(S^1)$ have ping-pong dynamics.     Then $a_1, ... , a_n$ generate a free group.  More precisely, for any $x_0$ not in the interior of an attracting domain $D(a_i^{\pm 1})$, the orbit of $x_0$ is free and its cyclic order is completely determined by: \emph{i)} the cyclic order of the sets $D_j(s)$ and $\{x_0\}$, for $s \in \{a_1^{\pm 1}, ... , a_n^{\pm 1}\}$, and \emph{ii)} the collection of \emph{containment relations}
$$s(D_k(t)) \subset D_j(s), $$
$$s(x_0) \in D_j(s)$$  
\end{lemma}

\begin{proof}[Proof of Lemma \ref{lem pingpong}]
Let $w_1$, $w_2$ and $w_3$ be distinct reduced words in the letters $a_i$ and $a_i^{-1}$.    
We need to show that the cyclic order of the triple $w_1(x_0), w_2(x_0), w_3(x_0)$ is well defined (i.e. $x_0$ has trivial stabilizer) and completely determined by the cyclic order of, and containment relations among, $x_0$ and the connected components of the attracting domains.  

We proceed by induction on the maximum word length of $w_i$.   For the base case, if each $w_i$ is either trivial or a generator, then either $w_i(x_0) = x_0$ (trivial case), or $w_i(x_0) \in D_j(w_i)$, for some $j$ determined by the containment relations.  Since the  $w_i$ are distinct, no two of the points $w_1(x_0), w_2(x_0)$ and $w_3(x_0)$ lie in the same domain $D(w_i)$.  Thus, all are distinct points, and the cyclic order of the sets $D_j(s)$ and $\{x_0\}$ determines the cyclic order of the triple.  

For the inductive step, assume that for all triples of reduced words $w_i$ of length at most $k$, the cyclic order of the triple $w_1(x_0), w_2(x_0), w_3(x_0)$ is determined, and assume also that the points $w_i(x_0)$ have $w_i(x_0) \in D_j(w_i)$ for some $j$ determined by $w_i$.   (For completeness and consistency, consider $\{x_0\}$ to be the attracting domain for the empty word).  
Let $w_1, w_2, w_3$ be reduced words of length at most $k+1$.  

Write $w_i = s_i v_i$, where $s_i \in \{a_1^{\pm 1}, ... a_n^{\pm 1}\}$, so $v_i$ is a word of length at most $k$.  (If $w_i$ is empty, skip this step, and set $v_i$ to be the empty word.)  We have $v_i(x_0) \in D_k(t_i)$ for some $k$, and since $w_i$ is a reduced word, $t_i \neq s_i$.  By inductive hypothesis, the sets $D_k(t_i)$ and cyclic order of the points $v_i(x_0)$ are known.  
Finally, we have $w_i(x_0) = s_iv_i(x_0) \in s_i(D_k(t_i))$, and $s_i(D_k(t_i)) \subset D_j(s_i)$ for some $j$ given by the containment relations.    

If these sets $D_j(s_i)$ are distinct, then we are done since their cyclic order is known.  We may also ignore the case where all $w_i$ have the same initial first letter $s$, since the cyclic order of the triple $s^{-1}w_1(x_0), s^{-1}w_2(x_0), s^{-1}w_3(x_0)$ agrees with that of $w_1(x_0), w_2(x_0), w_3(x_0)$.   So we are left with the case where exactly two of the three domains $D_j(s_i)$ containing $w_i(x_0)$ agree.  For concreteness, assume $w_1$ and $w_2$ start with $s$ and $w_3$ does not, and $w_1(x_0)$ and $w_2(x_0)$ lie in $D_j(s)$.  

Consider the shorter reduced words $v_1 = s^{-1}w_1$, $v_2 = s^{-1}w_2$ and $s^{-1}$.  By hypothesis, the cyclic order of the points $v_1(x_0), v_2(x_0), s^{-1}(x_0)$ is determined.  This order agrees with the order of the triple $w_1(x_0), w_2(x_0), x_0$.     As $x_0 \notin D_j(s)$ and $w_3(x_0) \notin D_j(s)$, this order is the same as that of $w_1(x_0), w_2(x_0), w_3(x_0)$.  Thus, the cyclic order of the triple $w_1(x_0), w_2(x_0), w_3(x_0)$ is determined by this initial configuration.  

\end{proof} 

The most basic and familiar example of ping-pong dynamics is as follows.  
\begin{example}  \label{pingpong const}
Let $D(a), D(a^{-1}), D(b),$ and $D(b^{-1})$ be disjoint closed intervals in $S^1$.   Let $a$ and  $b $ be orientation-preserving homeomorphisms of the circle such that 
\begin{align} \label{domain eq}
\begin{split}
&a(S^1\setminus D(a^{-1}))\subset D(a),\\
&b(S^1\setminus D(b^{-1}))\subset D(b).
\end{split} 
\end{align}  
By Lemma \ref{lem pingpong}, any point $x_0$ in the complement of the union of these attracting domains induces a circular order on $F_2$ by 
$c(w_1,w_2,w_3)= \ord(w_1(x_0),w_2(x_0),w_3(x_0)).$
\end{example}

\begin{remark} \label{separating rk}
Note that there are two dynamically distinct cases in the construction above: either we can take $D(b)$ and $D(b^{-1})$ to lie in different connected components of $S^1 \setminus \big( D(a) \cup D(a^{-1}) \big)$, or to lie in the same connected component of $S^1 \setminus \big( D(a) \cup D(a^{-1}) \big)$. 
\end{remark}

\begin{convention}
For the remainder of this section, whenever we speak of actions with ping-pong dynamics, we assume that each attracting domain $D(s)$ has \emph{finitely many} connected components.  
\end{convention}

The next Proposition shows that ping-pong dynamics come from isolated circular orders, proving one direction of Theorem \ref{domains thm}.

\begin{proposition}  \label{ping pong isolated}
Let $\rho: F_n \to \Homeo_+(S^1)$ be the dynamical realization of a circular order $c$, with basepoint $x_0$.  If $\rho$ has ping-pong dynamics and $x_0 \in S^1 \setminus \bigcup_i D(a_i^{\pm 1})$, then $c$ is isolated in $\CO(F_n)$. 
\end{proposition}  

\begin{proof}
Let $\rho: F_n \to \Homeo_+(S^1)$ be the dynamical realization of a circular order $c$, with basepoint $x_0$, and assume that $\rho$ has ping-pong dynamics.   
Following the convention above, let $D_1(s), D_2(s), ... $ denote the (finitely many) connected components of $D(s)$, so 
 for each $s \in \{a_1^{\pm 1}, ... a_n^{\pm 1}\}$, each $t \neq s^{-1}$ and each connected component $D_k(t)$, there exists $j$ and $i$ such that  we have containment relations of the form
\begin{equation} \label{ping eq}
\rho(s)(D_k(t)) \subset \overset{\circ}D_j(s) \\
\text{ and } \\
\rho(s)(x_0) \in \overset{\circ}D_i(s)
\end{equation} 
Here we use the notation $\overset{\circ}D_i(s)$ for the interior of $D_i(s)$, and in this proof we will use the fact that the image lies in the interior.   

We now show these dynamics are stable under small perturbations.  To make this precise, for  each $s \in \{a_1^{\pm 1}, ... a_n^{\pm 1}\}$ and each connected component  $D_k(s)$, let $D'_k(s)$ be an $\epsilon$--enlargement of $D_k(s)$, with $\epsilon$ chosen small enough so that $D_k(s) \subset \overset{\circ}{D'}_k(s)$, but all the domains $D'_k(s)$ remain pairwise disjoint.  
Now if $\rho'(s)$ is sufficiently close to $\rho(s)$, then we will have 
$$\rho'(s)(S^1 \setminus D'(s^{-1})) \subset \rho'(s)(S^1 \setminus D(s^{-1})) \subset D'(s).$$  Moreover, as in \eqref{ping eq}, we will also have
$\rho'(s)(D'_k(t)) \subset \overset{\circ}{D'}_j(s)$ and 
$\rho'(s)(x_0) \in \overset{\circ}{D'}_i(s)$ (with the same indices).  

Thus, there exists a neighborhood $U$ of $\rho$ in $\Hom(\Gamma, \Homeo_+(S^1))$ such that any $\rho' \in U$ is a ping-pong action for which the sets $D'_j(s)$ may be taken as the connected components of attracting domains.  Moreover, these components are in the same cyclic order as the components $D_j(s)$, the containments from equation \eqref{ping eq} are still valid, and $x_0$ is in the same connected component of the complement of the domains.  
Fix such a neighborhood $U$ of $\rho$.    

By Lemma \ref{lem pingpong}, the cyclic order of the orbit of $x_0$ depends only on the cyclic order of the domains $D'_j(a_i^{\pm 1})$ and $x_0$, so for each $\rho' \in U$, the cyclic order of $\rho'(F_n)(x_0)$ agrees with that of $\rho(F_n)(x_0)$. 
Now by Lemma \ref{lem continuity}, there exists a neighborhood $V$ of $c$ in $\CO(F_n)$ such that any order $c' \in V$ has a dynamical realization $\rho'$ with basepoint $x_0$ such that $\rho' \in U$.  As we observed above, the order of the orbit $\rho_c(F_2)(x_0)$ agrees with that of $\rho(F_2)(x_0)$, and so the two circular orders agree.
\end{proof} 

Although much less obvious, the converse to Proposition \ref{ping pong isolated} is also true: 

\begin{proposition} \label{iso gives ping-pong}
Suppose $\rho$ is a dynamical realization  of an isolated circular order on $F_n = \langle a_1, a_2, ..., a_n \rangle$.  Then there exist 
disjoint closed sets $D(s) \subset S^1$ for every $s \in \{a_1^{\pm1}, ... ,a_n^{\pm1}\}$, each consisting a finite union of intervals, and each disjoint from the basepoint, such that $\rho(s) (S^1 \setminus D(s^{-1})) \subset D(s)$ holds for all $s$.  
\end{proposition}

This proposition gives the other direction of Theorem \ref{domains thm}.  

\begin{proof}[Proof of Proposition \ref{iso gives ping-pong}]
Assume that $\rho$ is the dynamical realization of an isolated circular order on $F_n$, and let $x_0$ be the basepoint for $\rho$. 
We start by proving the following claim.   To simplify notation, here $F_n(x_0)$ denotes the orbit of $x_0$ under $\rho(F_n)$, and $\overline{F_n(x_0)}$ denotes its closure.  

\begin{claim*} For any $y \in \overline{F_n(x_0)} \setminus\{x_0\}$, there exists a neighborhood $U$ of $y$ in $S^1$, and $s \in \{a_1^{\pm1}, ... , a_n^{\pm1} \}$ such that, for each point $\rho(g)(x_0) \in U \cap F_n(x_0)$, the (reduced) word $g$ in $\{a_1^{\pm1}, ... , a_n^{\pm1}\}$ has the same initial letter $s$. 
\end{claim*} 

We prove this claim by contradiction.  If the claim is not true, then there is some set of the form 
$$\{ \rho(sw)(x_0) : sw \text{ is a reduced word in }a_1^{\pm1}, ..., a_n^{\pm1} \}$$ (where $s \in \{a_1^{\pm1}, ... , a_n^{\pm1}\}$ is fixed), with points of the form $\rho(tv)(x_0)$ $t \neq s$, arbitrarily close to its closure.  In other words, given any $\epsilon > 0$, we can find two points  $\rho(sw)(x_0)$ and $\rho(tv)(x_0)$, with $s \neq t$, that are distance at most $\epsilon/2$ apart.   Write $w = w_k w_{k-1}...w_1$ as a reduced word in the letters $a_1^{\pm1}, ..., a_n^{\pm1}$, and similarly write $v = v_l v_{l-1}...v_1$, and consider all the images of $x_0$ under initial strings of these words, i.e. the points $\rho(w_{k'}  ...w_1)(x_0)$ and $\rho(v_{l'}...v_1)(x_0)$ for $k' \leq k$ and $l' \leq l$.
We may assume that none of these points lie in the shorter than $\epsilon/2$-length interval between $\rho(sw)(x_0)$ and $\rho(tv)(x_0)$ -- otherwise, we may replace $sw$ and $tv$ with two of these initial strings, say $u$ and $u'$, that still have different initial letters, so that $\rho(u)(x_0)$ and $\rho(u')(x_0)$ have distance less than $\epsilon/2$ apart, and so that the images of $x_0$ under initial strings of $u$ and $u'$ do not lie in the small interval between $\rho(u)(x_0)$ and $\rho(u')(x_0)$.  

Now we modify $\rho$ by replacing $\rho(s)$ with $h \rho(s)$, where $h$ is a homeomorphism supported on a small neighborhood of the interval between $\rho(sw)(x_0)$ and $\rho(tv)(x_0)$.  Choose $h$ such that the triple $h \rho(s)\rho(w)(x_0), \rho(tv)(x_0), x_0$ has the opposite orientation from the triple $\rho(sw)(x_0), \rho(tv)(x_0), x_0$.   Additionally, we may take the interval where $h$ is supported to be small enough to not contain any point of the form $\rho(w_{k'}  ...w_1)(x_0)$ or $\rho(v_{l'}...v_1)(x_0)$.  
We leave the images of the other generators unchanged.    Call this new action $\rho'$.  

Even though $\rho(s) \neq \rho'(s)$ and $s$ may appear as a letter in $v$ or $w$, we claim that the triple $\rho'(sw)(x_0), \rho'(tv)(x_0), x_0$ does indeed have the opposite orientation from the triple $\rho(sw)(x_0), \rho(tv)(x_0), x_0$.   This is because the support of $h$ is disjoint from all points $\rho(w_{k'}  ...w_1)(x_0)$ and $\rho(v_{l'}...v_1)(x_0)$, so one can see inductively that in fact $\rho'(w_{k'}  ...w_1)(x_0) = \rho(w_{k'}  ...w_1)(x_0)$ and $\rho'(v_{l'}...v_1)(x_0) = \rho(v_{l'}...v_1)(x_0)$ for all $k' \leq k$ and $l' \leq l$.   It follows that $\rho'(sw)(x_0) = h \rho(sw)(x_0)$, and $\rho'(tv)(x_0) = \rho(tv)(x_0)$.  (It is easy to check that this even works if $t = s^{-1}$.)  Thus, by definition of $h$, the triples have opposite orientations.   

We have just shown that, for any $\epsilon > 0$, there exists an action of $F_n$ on $S^1$ such that the image of any point under any generator of $F_n$ is at most distance $\epsilon$ from the original action $\rho$, and yet the cyclic order of the orbit of $x_0$ under the new action is different.  Given any finite subset of $F_n$ we can choose $\epsilon$ small enough so that the new order on the orbit of $x_0$ will agree with the previous one on this finite set.  Order completion now gives an arbitrarily close circular order to the original, hence  could not have been isolated.    This completes the proof of the claim. 

Now we finish the proof of the proposition.  Assume that $\rho$ is the dynamical realization of an isolated circular order with basepoint $x_0$.  We construct the domains $D(s)$ in three steps.  \\
\textit{Step 1}. For each $s$, declare that $D(s)$ contains every point of the form $\rho(sw)(x_0)$, and every point $y$ in the closure of $\{ \rho(sw)(x_0) : sw \text{ a reduced word in } a_i^{\pm1} \}$.  The claim we just proved implies that accumulation points of $\{ \rho(sw)(x_0) : sw \text{ a reduced word in } a_i^{\pm1} \}$ are disjoint from those of $\{ \rho(tw)(x_0) : tw \text{ a reduced word in } a_i^{\pm1} \}$ for $s \neq t$.  
  \\
\textit{Step 2.} If $I$ is a connected component of the complement of $\overline{F_n(x_0)}$, and the endpoints of $I$ are both in the already constructed $D(s)$, then declare $I \subset D(s)$.   
Note that the sets $D(s)$ are pairwise disjoint, and so far we have $\rho(s)(D(t)) \subset D(s)$ for every $t \neq s^{-1}$.  
 \\
 \textit{Step 3.} As defined so far, the sets $D(s)$ cover all of $S^1$ \emph{except} for some intervals complementary to $\overline{F_n(x_0)}$, precisely, those intervals with boundary consisting of one point in $D(s)$ and the other point in $D(t)$ for some $s \neq t$.  (There is also an exceptional case where one of the endpoints is $x_0$, i.e. allowing for the empty word.)  
 Our first claim, i.e. that every point in $\overline{F_n(x_0)}$ has a neighborhood containing only points of $\{ \rho(sw)(x_0) : sw \}$ for some fixed $s$, implies, since $S^1$ is compact, that there are only finitely many such complementary intervals. Also, since $\rho$ is a dynamical realization, it is not possible to have a complementary interval $I$ with both endpoints in $\overline{F_n(x_0)}\setminus F_n(x_0)$ (see Construction \ref{embedding def}). Further, form Theorem \ref{isolated linear part}, it follows that the isolated circular order associated to $\rho$ has a non-trivial linear part which is isomorphic to $\Z$ (since higher rank free groups has no isolated left orders, see \cite{Navas orders}), thus there is an element  $h_{min}\in F_n$, such that the (small) interval $I_{min}=(x_0, h_{min}(x_0))$ has no point of $F_n(x_0)$ in its interior. The same holds for $g(I_{min})$ for any $g\in F_n$. This implies that any complementary interval of the sets $D(s)$ must have each endpoint inside $F_n(x_0)$.  

Suppose that $I$ is such a complementary interval, so $I$ is of the form $(s_1 s_2 ...s_k(x_0),  t_1 t_2, ... t_l(x_0))$ where $s_1 s_2 ...s_k$ and $t_1 t_2... t_l$ are reduced words in the generators $a_1^{\pm1}, ... , a_n^{\pm1}$ and $s_1 \neq t_1$.  
This implies that $s_1^{-1}(I) = (s_2 ...s_k(x_0),  s_1^{-1}t_1 t_2... t_l(x_0))$ is also a complementary interval since $s_1... s_k$ is a reduced word i.e. $s_2 \neq s_1^{-1}$.  Similarly, $t_1^{-1}(I) = (t_1^{-1}s_1 s_2 ...s_k(x_0),  t_2... t_l(x_0))$ is also a complementary interval.   
Note that, for any $u \notin \{s_1^{-1}, t_1^{-1}\}$ the endpoints of $u(I)$ are (reduced) words beginning in $u$, so $u(I) \subset D(u)$.  

Iterating this argument, it follows that, for any proper initial string $s_1... s_p$ with $p< k$, the interval $(s_1... s_p)^{-1}(I)$ is complementary, as are its images under $s_{p+1}^{-1}$ and $s_{p}$. The interval $(s_1... s_k)^{-1}(I) = (x_0, (s_1... s_k)^{-1}t_1...t_l(x_0))$ is also complementary, and also its image under $s_k$, but its images under every other generator are contained in some already defined $D(s)$.   An analogous statement holds for images of $I$ under initial strings of $t_1...t_l$.  
In particular, this reasoning shows that $h_{min} = (s_1... s_k)^{-1}t_1...t_l$, and that any complementary interval is equal to the image of $(x_0, h_{min}(x_0))$ under the inverse of a prefix of  $h_{min}$. For convenience, we switch to some simpler notation, writing $h_{min} = u_1 u_2 ... u_n$, so that the complementary intervals are
$$(x_0,u_1\ldots u_n(x_0)), \, (u_1^{-1}(x_0), u_2\ldots u_n(x_0)),\ldots,\, (u_n^{-1}\ldots u_1^{-1}(x_0),x_0).$$


We now show how to enlarge the domains $D(s)$ to contain parts of these complementary intervals, in order to have them satisfy the ``ping-pong condition" $\rho(s) (S^1 \setminus D(s^{-1})) \subset D(s)$, for each of the generators $s$.   We remind the reader that, for every other interval $J$ in (an already defined portion of) a set $D(s)$, and for every $t \neq s^{-1}$, we have $t(J) \subset D(t)$.  Thus, after extending domains we need only check what happens on the complementary intervals.  

First, choose $p_1\in (x_0,u_1\ldots u_n(x_0))$ and add $[p_1,u_1\ldots u_n(x_0)]$ to $D(u_1)$ and $[u_1^{-1}(x_0), u_1^{-1}(p_1)]$ to $D(u_1^{-1})$. Note that this is coherent with Step 1. Note also that in this way we have  that $u_1^{-1}(I_{min}\setminus D(u_1))\subset D(u_1^{-1})$ and that $u_1(u_1^{-1}(I_{min})\setminus D(u_1^{-1}))\subset D(u_1)$ (so in particular $s (I_{min}\setminus D(u_1))\subset D(s)$ and $s(u_1^{-1}(I_{min}\setminus D(u_1^{-1}))\subset D(s)$ for every generator $s$). Then, we choose $p_2\in (u_1^{-1}(p_1),u_2\ldots u_n(x_0))$ and add $[p_2,u_2\ldots u_n(x_0)]$ to $D(u_2)$ and $[u_2^{-1}u_1^{-}1(x_0),u_2^{-1}(p_2)]$ to $D(u_2^{-1})$. 

We repeat this procedure iteratively; in the $j^{th}$ step we choose $p_j\in (u_{j-1}^{-1}(p_{j-1}),u_j\ldots u_n (x_0))$ and add $[p_j,u_j\ldots u_n(x_0)]$ to $D(u_j)$ and $[u_j^{-1} \ldots u_1^{-1}(x_0), u_j^{-1}(p_j)]$ to $D(u_j^{-1})$. This procedure ends with the choice of $p_n\in (u_{n-1}^{-1}(p_{n-1}),u_n(x_0)) $ and the addition of $[p_n,u_n(x_0)]$ to $D(u_n)$ and the addition of $[u_n^{-1} \ldots u_1^{-1}(x_0), u_n^{-1}(p_n)]$ to $D(u_n^{-1})$. It is easy to check that the domains $D(s)$'s remain disjoint, and that now the ping pong condition is satisfied globally. \end{proof}


 

Fixing $n$, it is now relatively easy to produce infinitely many non-conjugate actions of $F_n$ on $S^1$ by varying the number of connected components and cyclic orientation of domains $D(s)$ for each generator $s$.  Moreover, these can be chosen such that taking the orbit of a point produces infinitely many nonconjugate circular orders on $F_n$  (the reader may try this as an exercise, otherwise we will see some explicit examples shortly).   However, the existence of such actions is \emph{not} enough to prove that $\CO(F_n)$ has infinitely many nonconjugate isolated points.  This is because not every ping-pong action arises as the dynamical realization of a circular order; therefore Proposition \ref{ping pong isolated} does not automatically apply.

As a concrete example, only one of the two cases in Remark \ref{separating rk} is the dynamical realization of a circular order.  We will soon see that a circular order on $F_2$ produced by Construction \ref{pingpong const} when $D(b)$ and $D(b^{-1})$ lie in different connected components of $S^1 \setminus \big( D(a) \cup D(a^{-1}) \big)$ is isolated, but one produced when $D(b)$ and $D(b^{-1})$ lie in the same connected component is not isolated.  

To illustrate this point, and as a warm up to the proof of Theorem \ref{f_2_iso_thm}, we now prove exactly which ping-pong actions come from dynamical realizations in the simple case where the domains $D(s)$ are all connected.  

\subsection{Schottky groups and simple ping-pong dynamics}

\begin{definition} 
Say that an action of $F_n$ on $S^1$ has \emph{simple} ping-pong dynamics if the generators satisfy the requirements of a ping-pong action, and there exist connected attracting domains $D(s)$ for all generators and inverses.  
\end{definition}

This short section gives a complete characterization of dynamical realizations with simple ping-pong dynamics.  To do this, we will use some elementary hyperbolic geometry and results on classical Schottky subgroups of $\PSL(2,\R)$.    Although our exposition aims to be self-contained,  a reader looking for more background can refer to \cite[Ch. 1 and 2]{Dalbo} for a good introduction.  

\paragraph{On ping-pong in $\PSL(2,\R)$.}
There is a natural action of $\PSL(2,\R)$ on $S^1 = \R \cup \{\infty\}$ by M\"obius transformations.  A finitely generated subgroup $G \subset \PSL(2,\R)$ is called \emph{Schottky} exactly when it has ping-pong dynamics.  
The benefit of working in $\PSL(2,\R)$ rather than $\Homeo_+(S^1)$ is that M\"obius transformations of the circle extend canonically to the interior of the disc.  Considering the interior of the disc as the Poincar\'e model of the hyperbolic plane, Schottky groups act properly discontinuously by isometries.  Thus, it makes sense to describe the hyperbolic surface obtained by quotient the disc by  such a ping-pong action. We will prove the following.  

\begin{theorem} \label{n even thm}
Let $c \in \CO(F_n)$.  If the dynamical realization $\rho_c$ of $c$ has simple ping-pong dynamics, then $n$ is even, and $\rho_c$ is topologically conjugate to a representation $\rho: F_n \to \PSL(2, \R) \subset \Homeo_+(S^1)$ corresponding to a hyperbolic structure on a genus $n/2$ surface with one boundary component.  
\end{theorem}

Note that the conclusion of the theorem is (as it should be) independent of any choice of generating set for $F_n$, even though the definition of ping-pong dynamics is phrased in terms of a specific set of generators.  

As a special case, Theorem \ref{n even thm} immediately gives many concrete examples of actions of free groups on $S^1$ that \emph{do not} arise as dynamical realizations of any circular order (c.f. Lemma \ref{more boundary lem}) and justifies the remarks at the end of the previous subsection.  
\medskip

The main idea of the proof of Theorem \ref{n even thm} can be summarized as follows: 
For $\PSL(2,\R)$ actions, the condition that the quotient has a single boundary component exactly captures the condition that $F_n$ acts transitively on the connected components of the complement of the exceptional minimal set, i.e. the condition of Corollary \ref{min set 2}.   For general ping-pong actions in $\Homeo_+(S^1)$, we use the ``minimality" property of being a dynamical realization (Proposition \ref{semiconj prop}) to produce a conjugacy into 
$\PSL(2,\R)$, then cite the $\PSL(2,\R)$ case.  

As motivation and as a first step in the proof, we start with an example and a non-example. 

\begin{lemma}[Example of dynamical realization]  \label{one boundary lem}
Let $a_1, b_1, a_2, b_2,  ... , a_n, b_n$ denote generators of $F_{2n}$. 
Let $\rho: F_{2n} \to \PSL(2,\R) \subset \Homeo_+(S^1)$ have simple ping-pong dynamics.  Suppose that $x_0$ and the attracting domains are in the cyclic (counterclockwise) order 
$$ x_0, \, D(a_1), \, D(b_1), \, D(a_1^{-1}), \, D(b_1^{-1}),  \, D(a_2), \, D(b_2), ... , D(a_n^{-1}), \, D(b_n^{-1}).$$
Then $\rho$ is the dynamical realization of a circular order with basepoint $x_0$.  
\end{lemma}

Note that the $n=1$ case is one of the cases from Remark \ref{separating rk}.

\begin{proof} 
By the ping-pong lemma, $x_0$ has trivial stabilizer under $\rho(F_{2n})$, so its orbit defines a circular order $c$ on $F_{2n}$.  We claim that the dynamical realization of this circular order is $\rho$.  To see this, we first describe the action of $\rho$ more concretely using some elementary hyperbolic geometry.   Having done this, we will be able to apply Corollary \ref{min set 2} to the action.  

The arrangement of the attracting domains specified in the lemma implies that the quotient of $\H^2$ by $\rho(F_{2n})$ is, topologically, the interior of a surface $\Sigma$ of genus $n$ with one boundary component.  Geometrically, it is a surface of infinite volume with a singe end, as illustrated in Figure \ref{coverfig} for the case $n = 1$.  (This is elementary; see Proposition I.2.17 and discussion on page 51 of \cite{Dalbo} for more details.) 

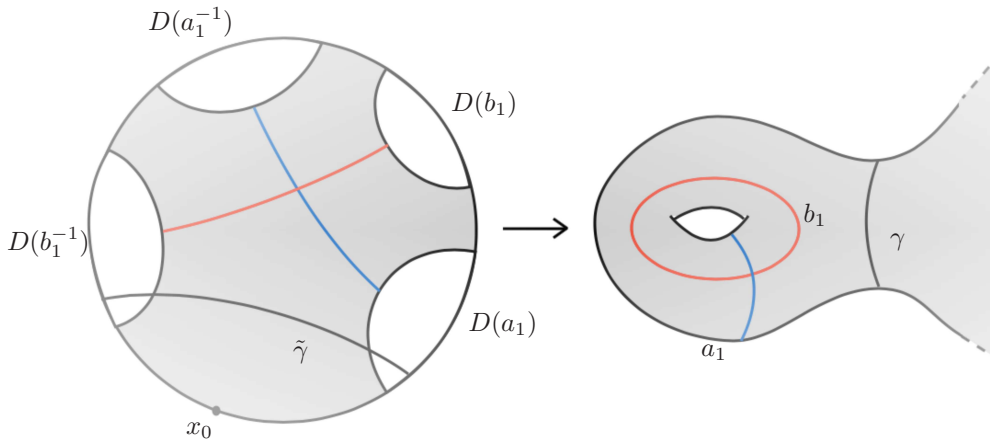
\begin{figure}
\centering
\labellist
\small\hair 2pt
\pinlabel $x_0$ at 140 15
\pinlabel $D(a_1)$ at 470 130
\pinlabel $D(b_1)$ at 450 370
\pinlabel $D(a_1^{-1})$ at 130 460
\pinlabel $D(b_1^{-1})$ at -25 220
\pinlabel $\tilde{\gamma}$ at 250 100
\pinlabel $\gamma$ at 900 220
\pinlabel $a_1$ at 700 100
\pinlabel $b_1$ at 810 250
\endlabellist
\includegraphics[width=5in]{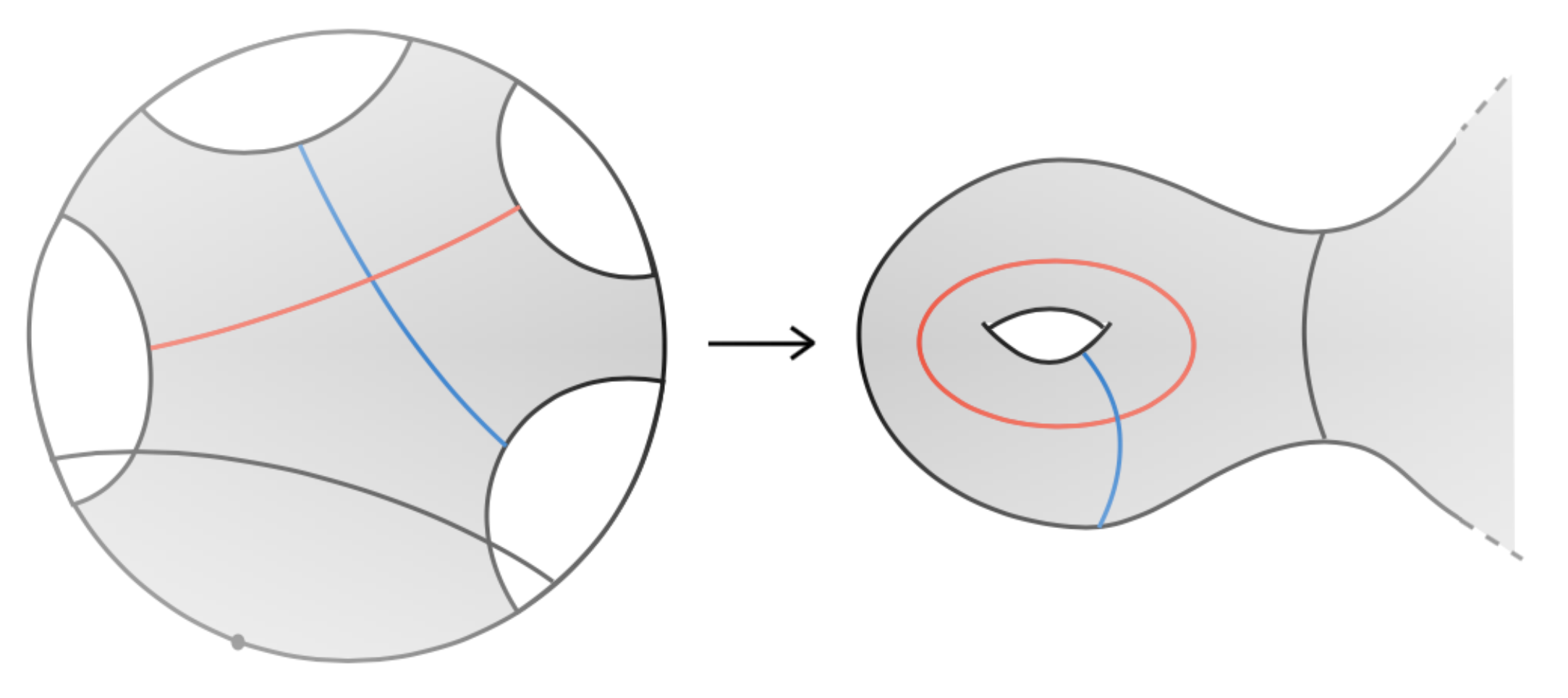}     
\caption{The open disc as universal cover $\tilde{\Sigma}$. The shaded area is a fundamental domain}  
\label{coverfig}
\end{figure}

Let $\gamma \subset \Sigma$ be a simple geodesic curve isotopic to the boundary.  The complement of $\gamma$ in $\Sigma$ has two connected components, a compact genus $n$ surface with boundary $\gamma$ (call this surface $\Sigma'$), and an infinite volume annulus, call this $A$.   Fixing a basepoint on $\Sigma$ and considering $\gamma$ as a based curve lets us think of it as an element of $\pi_1(\Sigma) = F_{2n}$, and its image $\rho(\gamma) \in \PSL(2,\R)$ is a hyperbolic element translating along an axis, $\tilde{\gamma} \subset \H^2$, which projects to the curve $\gamma \subset \Sigma$. 

We now describe the exceptional minimal set for $\rho$ in terms of the geometry of this surface.  It is a standard fact that the exceptional minimal set is precisely the limit points of the universal cover of the compact part $\tilde{\Sigma'}$ (see \cite[Sec. 3.1, Prop 3.6]{Dalbo}).   This can also be described by looking at images of $\tilde{\gamma}$ under the action of $\rho(F_{2n})$.  These images are disjoint curves, each bounding on one side a fundamental domain for the compact surface $\Sigma'$, and on the other a half-plane that contains a fundamental domain for the annulus $A$.   The intersection of the half-planes that cover $A$ with the $S^1$ boundary of $\H^2$ make up the complement of the exceptional minimal set for $\rho(F_{2n})$.  

What is important to take from this discussion is that there is a \emph{single orbit} for the permutation action of $\rho(F_{2n})$ on the complementary intervals to its exceptional minimal set, and this corresponds to the fact that $\Sigma$ has a single boundary component, in our case bounded by $\gamma$.   
As the exceptional minimal set is contained in the interior of the union of the attracting domains $D(a_i^{\pm 1})$ and $D(b_i^{\pm 1})$, the basepoint $x_0$ for the dynamical realization lies in a complementary interval $I$.  In our case, taking $\gamma$ to be the based curve represented by the commutator $[a,b]$,\footnote{here we use the convention for the action that $[a,b](x) = b^{-1}a^{-1}ba(x)$} our specification of the cyclic order of $x_0$ and the attracting domains were chosen so that $x_0$ lies in the interval bounded by the lift $\tilde{\gamma}$ that is the axis of this commutator.  The configuration is summarized in  Figure \ref{coverfig}.    Moreover, $\rho(\gamma)$ preserves $I$, acts on it by translations, and generates the stabilizer of $I$ in $\rho(F_{2n})$.  

In summary, $S^1$ is the union of an exceptional minimal set contained in the closure of $\rho(G)(x_0)$ and the orbit of the open interval $I$.  As the action of $\langle \gamma \rangle$ on $I$ is the dynamical realization of a left order, Corollary \ref{min set 2} implies that $\rho$ is the dynamical realization of $c$.  
\end{proof}

\medskip 

By contrast, geometric representations from surfaces with more boundary components do not arise as dynamical realizations.  

\begin{lemma}[Non-example] \label{more boundary lem}
Let $\rho: F_n \to \PSL(2,\R) \subset \Homeo_+(S^1))$ be such that $\H^2 / \rho(F_n)$ is a surface with more than one infinite volume (i.e. non-cusped) end.  Then $\rho$ is not the dynamical realization of a circular order on $F_n$.  
\end{lemma} 

A particular case of the lemma is the ping-pong action of $F_2 = \langle a, b \rangle$ with attracting domains in cyclic order $D(a), D(a^{-1}), D(b), D(b^{-1})$.  There the quotient $\H^2 / \rho(F_2)$ is homeomorphic to a sphere minus three closed discs. 

\begin{proof}  
Let $\rho$ be as in the Lemma, and let $\Sigma =  \H^2 / \rho(F_n)$.  Let $\gamma_1, \gamma_2, ... $ be geodesic simple closed curves homotopic to the boundary components of $\Sigma$.   As in the proof of Lemma \ref{one boundary lem}, $\rho(F_n)$ has an exceptional minimal set, and the connected components of the complement of this set are the boundaries of disjoint half-planes, bounded by lifts of the curves $\gamma_i$ to geodesics in $\H^2$.   For fixed $i$, the union of these half-planes bounded by lifts of $\gamma_i$ is a $\rho(F_n)$-invariant set.  

Since we are assuming that there is more than one boundary component, for any candidate for a basepoint $x_0$, there is some $i$ such that $x_0$ is \emph{not} contained in an interval bounded by a lift of $\gamma_i$.   Equivalently, there is a connected component of the complement of the exceptional minimal set that does not contain a point in the orbit of $x_0$.   By Lemma \ref{min set condition}, this implies that $\rho$ cannot be a dynamical realization of a circular order with basepoint $x_0$.  
\end{proof}

\begin{remark}
Given $\rho$ as in Lemma \ref{more boundary lem}, and a point $x \in S^1$ with trivial stabilizer, we get a circular order on $F_n$ from the orbit of $x$.   The lemma simply says that $\rho$ is \emph{not} its dynamical realization.  But it is not hard to give a positive description of what the dynamical realization actually is:  it is the action obtained from $\rho$ by collapsing every interval in the complement of the exceptional minimal set to a point, \emph{except} those that contain a point in the orbit of $x$.  
The resulting action is conjugate to a representation $\rho': F_n \to \PSL(2,\R)$ such that $\H^2/\rho'(F_n)$ is a surface with either all (or all but one) infinite-volume end of the surface $\H^2/ \rho(F_n)$ replaced by a finite-volume \emph{cusp}  Put otherwise, either for all $i$ or all but one $i$, we have that $\rho'(\gamma_i)$ is parabolic rather than hyperbolic.    Both cases do in fact occur; for instance, the cusp-only case arises by taking the orbit of a point with trivial stabilizer under $\rho'$, since such an action $\rho'$ is already minimal Proposition \ref{semiconj prop} implies that it is the dynamical realization of this order.  
We leave the details as an exercise for the reader.  

One can also show that such an action $\rho'$ is not rigid, by finding an arbitrarily small deformation of $\rho'$ so that a parabolic element $\rho'(\gamma_i)$ becomes an infinite order elliptic.  This new action is no longer semi-conjugate to the original -- in fact, even the circular order of the orbit of $x$ under the subgroup generated by $\gamma_i$ will change.  
\end{remark}

Now we can prove Theorem \ref{n even thm}; the goal in the proof is to reduce an arbitrary action to the examples from the $\PSL(2,\R)$ case considered in the previous two lemmas.

 \begin{proof}[Proof of Theorem \ref{n even thm}] 
 Let $\rho: F_n \to \Homeo_+(S^1)$ be the dynamical realization of a circular order, with basepoint $x_0$.  Assume that $\rho$ has simple ping-pong dynamics.  Let $D(a_i ^{\pm 1})$ be the attracting domains for the generators $a_i$ and inverses $a_i^{-1}$.  

Define a representation $\rho': F_n \to \PSL(2, \R) \subset \Homeo_+(S^1)$ by setting $\rho'(a_i)$ to be the unique hyperbolic element such that $\rho'(a_i)(D(a_i)) = \rho(a_i)(D(a_i))$.  
Then $\rho'$ is a simple ping-pong action, and we can take $D'(a_i^{\pm 1}) = D(a_i^{\pm 1})$ to be the attracting domains for $\rho'$.  
We will now show that 
\begin{enumerate}[label=\roman*)] 
\item there exists $f \in \Homeo_+(S^1)$ such that $f \circ \rho'(g) \circ f^{-1} = \rho(g)$ for all $g \in F_n$, and
\item the quotient of $\H^2$ by $\rho'$ is a genus $n/2$ surface with one boundary component, in particular, $n$ is even. 
\end{enumerate}
This will suffice to prove the theorem.  

By Lemma \ref{lem pingpong}, the cyclic order of $\rho(F_n)(x_0)$ and $\rho'(F_n)(x_0)$ agree, so the map $f:\rho(F_n)(x_0) \to \rho'(F_n)(x_0)$ given by $f(\rho(g)(x_0)) = \rho'(g)(x_0)$ is cyclic order preserving.  We claim that, similarly to Proposition \ref{semiconj prop}, $f$ extends continuously to the closure of $\rho(F_n)(x_0)$.   To see this, we look at successive images of attracting domains, following \cite[Prop. 1.11]{Dalbo}.   

If $w= s_1 s_2 ... s_k$ is a reduced word with $s_i \in \{a_1^{\pm 1}, ... a_n^{\pm 1}\}$ define the set $D(s_1, s_2, ..., s_k)$  by  $D(s_1, ...,  s_k) := \rho(s_1 ... s_{k-1}) D(s_k)$.  Make the analogous definition for $D'(s_1, ..., s_{k-1}) := \rho'(s_1, ..., s_{k-1}) D'(s_k)$.   Note that $D(s_1, s_2, ..., s_k) \subset \overset{\circ}D(s_1, s_2, ... s_{k-1})$.  

Let $g_k \in F_n$ be a sequence of distinct elements such that $\lim \limits_{i \to \infty} \rho(g_k)(x_0) = y$.  Write each $g_k$ as a reduced word in the generators and their inverses.  Since the set of generators and inverses is finite, after passing to a subsequence $g_{k_i}$, we may assume that there exist $s_1, s_2, ... $ with $s_{i+1} \neq s_i^{-1}$ and such that $g_{k_i} = s_1 s_2 ... s_i$.   Thus, 
$$\rho(g_{k_i})(x_0) \in D(s_1, s_2, ... ,s_i) \text{ and }$$
$$y \in \bigcap_{i=1}^\infty D(s_1, s_2, ... ,s_i).$$  
Returning to our original sequence, as $D(s_1, s_2,..., s_i) \subset \overset{\circ}D(s_1, s_2,..., s_{i-1})$, and $\rho(g_k)(x_0) \to y$, for any fixed $i$, we will have $\rho(g_k)(x_0) \in D(s_1, s_2,..., s_i)$ for all sufficiently large $k$. 

Since the containment relations among the domains $\rho(s_j)(D(s_i))$ and $\rho'(s_j)(D'(s_i))$ agree, it follows that $\rho'(g_k)(x_0) \in D'(s_1, s_2,... ,s_i)$ whenever $\rho(g_k)(x_0) \in D(s_1, s_2,..., s_i)$.   Since $\rho'$ is a ping-pong action with image in $\PSL(2,\R)$, it follows from hyperbolic geometry (see Lemma 1.10 in \cite{Dalbo}) that the intersection $\bigcap_i D'(s_1, s_2, ..., s_i)$ is a single point, say $y'$.  Thus, $\rho'(g_k)(x_0)$ converges to $y'$, and we may define $f(y) = y'$, giving a continuous extension of $f$.  

Since $\rho$ is a dynamical realization, as in the proof of Lemma \ref{min set condition}, Proposition \ref{semiconj prop} implies that $f$ has a continuous inverse and so must be a homeomorphism onto its image.  We may then extend $f$ over each complementary interval to the orbit of $x_0$, to produce a homeomorphism of $S^1$ such that $f \rho'(g) f^{-1} = \rho(g)$ for all $g \in F_n$.

Topologically, the quotient of $\H^2$ by $\rho'(F_n)$ is the interior of an orientable surface $\Sigma$ with $\pi_1(\Sigma) = F_n$, and therefore genus $g$ and $b = n-2g+1$ boundary components, for some $g$.  
Because $\rho'$ has simple ping-pong dynamics, it has at least one boundary component.  We showed in Lemma \ref{more boundary lem} that, if $\rho'$ is a dynamical realization, then there is at most one boundary component.  Since we are assuming that $\rho$, and hence its conjugate $\rho'$ is a dynamical realization, there is exactly one boundary component, and so $n = 2g$ is even.  This proves assertion ii.   
\end{proof}

\subsection{Infinitely many nonconjugate circular orders}
Finally, we give the proof that $\CO(F_{2n})$ has infinitely many nonconjugate isolated points.   These will come from ping-pong actions with disconnected domains that are ``lifts" of the Schottky actions described in the previous section.  

 \begin{proof}[Proof of Theorem \ref{f_2_iso_thm}]
 Let $G = F_{2n}$, and let $\rho: G \to \PSL(2, \R)$ be a representation such that $\H^2 / \rho(G)$ is an infinite volume genus $n$ surface with one end.  Fix a generating set $a_1, b_1, ... , a_n, b_n$ for $G$.    
 
 \begin{definition}
Fix $k > 1$.   The \emph{standard $k$-lift} $\hat{\rho}$ of $\rho$ is a representation defined as follows:  for each generator $s$ let $\hat{\rho}(s)$ be the unique lift of $\rho(s)$ to the $k$-fold cyclic cover of $S^1$ such that $\hat{\rho}(s)$ has fixed points.   This determines a homomorphism $\hat{\rho}: G \to \Homeo_+(S^1)$ whose image commutes with the order $k$ rigid rotation.   
 \end{definition} 

Let $D(a_i^{\pm 1})$, $D(b_i^{\pm 1})$ be the attracting domains for the generators and their inverses with respect to the representation $\rho$.    For each generator or inverse $s$, let $\hat{D}(s)$ be the pre-image of an attracting $D(s)$ under the ($k$-fold) covering map $\pi: S^1 \to S^1$.  Then $\hat{D}(s)$ is a (disconnected) attracting domain for $\hat{\rho}(s)$, and $\hat{\rho}$ has ping-pong dynamics.  Let $\hat{x}_0$ be a lift of $x_0$.  The next lemma shows that, for infinitely many choices of $k$, this lift $\hat{\rho}$ is the dynamical realization of a circular order with basepoint $\hat{x}_0$.  Having done this, Proposition \ref{ping pong isolated} implies that this circular order is isolated in $\CO(F_{2n})$.  

\begin{lemma}
If $k$ and $2n-1$ are relatively prime, then the standard $k$-lift $\hat{\rho}$ constructed above is the dynamical realization of a circular order.  
\end{lemma}
 
\noindent \textit{Proof.} 
 Again, we begin by describing the dynamics of $\hat{\rho}$.   All of the facts stated here follow easily from the construction of $\hat{\rho}$ as a lift of $\rho$.   Further explanation and a more detailed description of such lifts can be found in Section 2 of \cite{Mann invent}.  
 
Let $K$ be the exceptional minimal set for $\rho$.  Then $\hat{K} := \pi^{-1}(K)$ is the exceptional minimal set for $\hat{\rho}$, and $\hat{x}_0 \in S^1 \setminus \hat{K}$.  Let $I$ denote the connected component of $S^1 \setminus K$ containing $x_0$, and $\hat{I}$ the connected component of $S^1 \setminus \hat{K}$ containing $\hat{x}_0$.  
Using Corollary \ref{min set 2}, it suffices to show that $\hat{\rho}(G)$ acts transitively on the set of connected components of $S^1 \setminus \hat{K}$ and that the stabilizer of $\hat{I}$ acts on $\hat{I}$ as a dynamical realization of a linear order.  

Let $\hat{I} = I_0$, and let $I_1, I_2, ... , I_{k-1}$ denote the other connected components of $\pi^{-1}(I)$, in cyclic (counterclockwise) order.  

Let $g \in F_{2n}$ be the product of commutators $g= [a_1, b_1]...[a_n, b_n]$, so $\rho(g)$ is the stabilizer of $I$ in $\rho(G)$.   For the lifted action, we have
$\hat{\rho}(g)(I_i) = I_j$, where $j = i + 2n-1$ mod $k$.  (This follows from the fact that the \emph{rotation number} of $\hat{\rho}(g)$ is $(2n-1)/k$, see \cite{Mann invent}.)   Hence, the stabilizer of $\hat{I}$ is the infinite cyclic subgroup generated by $g^k$.   Moreover, $\hat{\rho}(g^k)$ acts on $\hat{I}$ without fixed points, as it is a lift of $\rho(g^k)$ which acts on $I$ without fixed points.   Thus, the action of the subgroup generated by $g^k$ acts on $\hat{I}$ as the dynamical realization of a left order on $\Z$. 

It remains only to show that $\hat{\rho}(G)$ acts transitively on the set of connected components of $S^1 \setminus \hat{K}$ so that we can apply Corollary \ref{min set 2}.  This is where we use the hypothesis that $k$ and $2n-1$ are relatively prime.   Assuming that they are relatively prime, the fact that $\hat{\rho}(g)(I_i) = I_j$, where $j = i + 2n-1$ mod $k$ will now imply that $\hat{\rho}$ acts transitively on the lifts of $I$.   In detail, if $\hat{J}$ is any other connected component of $S^1 \setminus \hat K$, with $\pi(\hat{J}) = J$, then we may find $f \in F_{2n}$ such that $\rho(f)(I) = J$.  It follows that $\hat{\rho}(g^m f)(\hat{I}) = \hat{J}$ for some $m$, so the action is transitive.  This completes the proof of the lemma. 
\qed

To finish the proof of the theorem, we need to show that different choices of $k$ give infinitely many distinct conjugacy classes of circular orders.   This follows from the fact we mentioned above that the rotation number of $\hat{\rho}(g)$ is $(2n-1)/k$.  Rotation number is a (semi)-conjugacy invariant of homeomorphisms of $S^1$.  Hence, the dynamical realizations of conjugate orders cannot assign different rotation numbers to the same element.  

 \end{proof}

\subsection{Explicit singleton neighborhoods of isolated orders} \label{neighborhood sec}

Recall that a neighborhood basis of an order $c \in \CO(G)$ is given by sets of the form $O_S(c): = \{ c' \in \CO(G) : c'(u,v,w) = c(u,v,w) \text{ for all } u, v, w \in S \}$, where $S$ ranges over all finite subsets of $G$.   Given an isolated circular order $c$, it is therefore very natural to ask: \emph{what is the minimum cardinality of $S$ such that $O_S(c)$ is a singleton?}  

In this section, we take the isolated circular order on $F_2 = \langle a, b \rangle$ whose dynamical realization has simple ping-pong dynamics (from Lemma \ref{one boundary lem}) and give an upper bound on the cardinality of such a set $S$, by exhibiting a specific set with 5 elements.  (We expect the bound $|S| \leq 5$ given here is sharp, but have not pursued this point.)
The proof of the Proposition below also gives an independent and very short proof that $\CO(F_2)$ has isolated points; the only previous tool that we use here is the ping-pong lemma.  

Fixing notation, let $c$ be the isolated circular from Lemma \ref{one boundary lem} in the case $n=1$.    We show the following: 

\begin{proposition} 
If $c'$ agrees with $c$ on the set $S^3=\{ \id, a, ba,  a^{-1}ba,  b^{-1}a^{-1}ba \}^3$, then $c' = c$.  
\end{proposition} 

\begin{proof}
Let $\rho: F_2 \to \Homeo_+(S^1)$ be the dynamical realization of $c$ with basepoint $x_0$.  We need to show that for any action $\rho'$ of $F_2$ on $S^1$ such that the cyclic order of $\rho'(S)(x_0)$ agrees with that of $\rho(S)(x_0)$, the cyclic order of $\rho'(G)(x_0)$ then agrees with that of $\rho(G)(x_0)$.    Consider an action of $F_2$ on $S^1$ with the points 
$$x_0, \,\, ab^{-1}a^{-1}b(x_0), \,\,  b(x_0),  \,\,  a^{-1}b(x_0), \,\,  b^{-1}a^{-1}b(x_0)$$ 
in this cyclic (counterclockwise) order.   
Here, and in the remainder of the proof, we drop the notation $\rho'$ for the action.  

We begin by choosing some additional points in $S^1$ in a careful way, so as to arrive at the configuration illustrated in Figure \ref{pingpongfig} below.  This will let us define attracting domains for the action.

Let $p_1 \in S^1$ be a point such that the triple $x_0, p_1, ab^{-1}a^{-1}b(x_0)$ is positively oriented.   Choose $x'_0$ very close to $x_0$, and such that $x'_0, x_0, p_1$ is positively oriented.  If $x'_0$ is chosen sufficiently close to $x_0$, then 
$$x'_0, \,\, p_1, \,\, ab^{-1}a^{-1}b(x'_0), \,\,  b(x'_0),  \,\,  a^{-1}b(x'_0), \,\,  b^{-1}a^{-1}b(x'_0)$$
will also be in positive cyclic order.    Let $X$ denote this set of 6 points.  

The ordering of $X$ implies that $a^{-1}(p_1)$ lies in the connected component of $S^1 \setminus X$ bounded by $a^{-1}b(x'_0)$ and $b^{-1}a^{-1}b(x'_0)$.  Let $p_2$ be a point in this component such that the triple $a^{-1}(p_1), p_2, b^{-1}a^{-1}b(x'_0)$ is positively oriented.  Similarly, our assumption on order implies that $b(p_2)$ lies in the component of $S^1 \setminus X$ bounded by $b(x'_0)$ and $a^{-1}b(x'_0)$, and we let $p_3$ be a point in this component such that $b(p_2), p_3, a^{-1}b(x'_0)$ is positively oriented.   This configuration is summarized in Figure \ref{pingpongfig}.

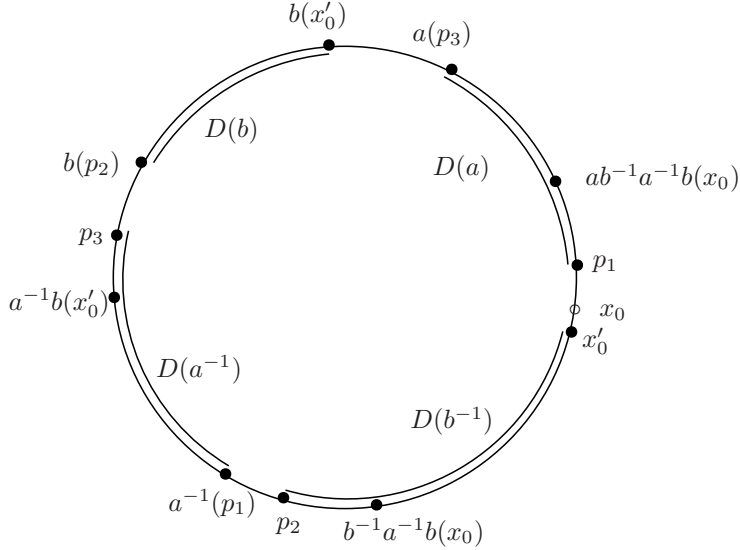
\begin{figure} 
\centering
\labellist
\small\hair 2pt
\pinlabel $b^{-1}a^{-1}b(x_0)$ at 160 -10
\pinlabel $p_2$ at 95 -7
\pinlabel $a^{-1}(p_1)$ at 55 5
\pinlabel $a^{-1}b(x'_0)$ at -25 110
\pinlabel $p_3$ at -8 144
\pinlabel $b(p_2)$ at -8 182
\pinlabel $b(x'_0)$ at 110 260
\pinlabel $a(p_3)$ at 175 250
\pinlabel $ab^{-1}a^{-1}b(x_0)$ at 290 175
\pinlabel $p_1$ at 260 130
\pinlabel $x'_0$ at 255 90
\pinlabel $\circ \ \ x_0$ at 256 105
\pinlabel $D(b)$ at 65 200
\pinlabel $D(b^{-1})$ at 180 48
\pinlabel $D(a^{-1})$ at 48 75
\pinlabel $D(a)$ at 185 180
\endlabellist
\includegraphics[width=2.5in]{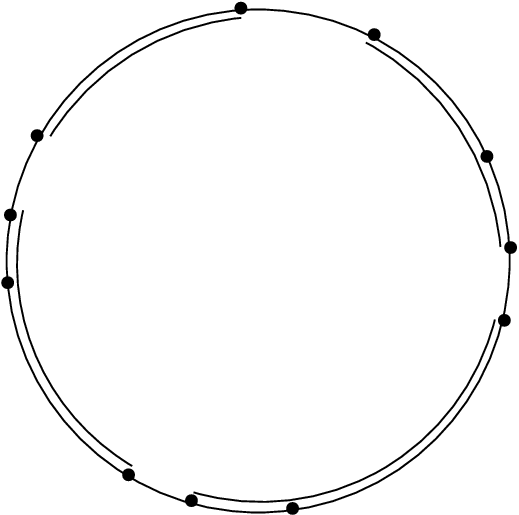}     

\caption{Configuration of points giving ping-pong domains on $S^1$} \label{pingpongfig}

\end{figure}

Now let $D(a)$, $D(b)$, $D(a^{-1})$ and $D(b^{-1})$ be the disjoint intervals $[p_1,a(p_3)]$,  $[b(x'_0), b(p_2)]$, $[p_3, a^{-1}(p_1)]$ and $[p_2, x'_0]$ as shown in  Figure \ref{pingpongfig}.    Our choice of configuration of points implies that $b(S^1 \setminus D(b^{-1})) \subset D(b)$.  Similarly,  $a( S^1 \setminus D(a^{-1})) \subset D(a)$.  We also have that $x_0$ is in the complement of the union of domains.  

Since these are the same cyclic order and containment relations as $c$, Lemma \ref{lem pingpong} implies that the cyclic order on $G$ induced by the orbit of $x_0$ coincides with $c$.  
\end{proof}

\begin{remark}
A similar strategy can be used to produce neighborhoods of the lifts of these orders on $F_2$, however, many more points are needed. 
Compare also the framework of \cite{Matsumoto new}, especially Lemma 4.8. 
\end{remark}

\section{Applications to linear orders}  \label{LO sec}

Suppose $0 \to A \to B \to C \to 0$ is a short exact sequence of groups.  
If $A$ is left-ordered, and $C$ circularly ordered, then it is well known that there is a natural circular order on $B$ such that both maps $A \to B$ and $B \to C$ are monotone.  (See \cite[Lemma 2.2.12]{Calegari} for a proof.)   Here, we discuss a different method of constructing orders on certain central extensions.

Recall that a subgroup $H \subset G$ is cofinal for a left order $<$ on $G$ if, for all $g \in G$, there exist $h_1, h_2 \in H$ such that $h_1 < g < h_2$.  
If $H \subset G$, we let $\LO_H(G) \subset \LO(G)$ denote the subspace of $\LO(G)$ consisting of orders where $H$ is cofinal.  

\begin{proposition}[see also \cite{Zheleva}] \label{cofinal prop}
Let $G$ be a group, and let 
$\Z \to \hat{G} \overset{\pi}\rightarrow G$ be a central extension.  Then, there is a continuous map $\pi^\ast=\pi^\ast_{\hat G}: \LO_\Z(\hat{G}) \to \CO(G)$.  Moreover, each circular order on $G$ is in the image of one such map $\pi^\ast$.  
\end{proposition}

This theorem is essentially proved in \cite{Zheleva}, although there is no comment on continuity there.

\begin{proof}
Let $\Z  \to \hat{G} \overset{\pi}\rightarrow G$ be a central extension.
Given $< $ in  $\LO_\Z(\hat{G})$ define $\pi^*(<)$ as follows.  Let $z$ be the generator of $\Z$ such that $z > \id$.  Since $\Z$ is cofinal, for each $g \in G$, there exists a unique $\hat{g} \in \hat{G}$ such that $\id \leq \hat{g} < z$.  Given distinct elements $g_1, g_2, g_3 \in G$, let $\sigma$ be the permutation such that $\id \leq \hat{g}_{\sigma(1)} < \hat{g}_{\sigma(2)} < \hat{g}_{\sigma(3)} < z$.   Define $\pi^*(<)(g_1, g_2, g_3) := \mathrm{sign}(\sigma)$.  One checks that this is a well defined circular order on $G$.  

To show continuity, given a finite set $S \subset G$, let $\hat{S}:= \{ \hat{g} \in \hat{G} : g \in S\}$.  If $<_1$ and $<_2$ are two left orders that agree on $\hat{S} \cup \{ \id, z\}$, then the definition of $\pi^*$ ensures that $\pi^*(<_1)$ and $\pi^*(<_2)$ agree as circular orders on $S$.

For the last remark, we give a proof for countable groups that highlights the relationship between the dynamical realization of $c$ and $\pi^*(c)$.  The general case is given in \cite{Zheleva}, and uses essentially the same idea.  Let $\Homeo_\Z(\R)$ denote the group of orientation-preserving homeomorphisms of $\R$ that commute with integer translations.  This group is the universal central extension of $\Homeo_+(S^1)$.  
Given $c \in \CO(G)$, let $\rho$ be a dynamical realization of $c$ with basepoint $x$, and let $\hat{x} \in \R$ be a lift of $x$.  Let $\hat{G}$ be the pullback of the central extension $0 \to \Z \to \Homeo_\Z(\R) \to \Homeo_+(S^1) \to 1$ using $\rho$.  It is easily checked that $\hat{x}$ has a free orbit under $\hat{G} \subset \Homeo(\R)$ so induces a left order $<$ on $\hat{G}$, and in this left order $\Z$ is cofinal.   By construction, $\pi^*(<) = c$. 
\end{proof} 

We note that, as remarked to the authors by S. Matsumoto, $\pi^\ast$ is not necessarily one-to-one, although the last step in the proof does give a partial inverse.  

\begin{lemma} \label{open lem}
Let $G$ be a finitely generated group, and $z \in G$ a central element.  The set of left invariant orders on $G$ where $\langle z\rangle $ is cofinal is open in $\LO(G)$.  
\end{lemma}

\begin{proof} 
Suppose that $<$ is a left order where $\langle z\rangle $ is cofinal.  Then for each generator $g_i$ in a finite generating set, there exists $k_i \in \Z$ such that $z^{k_i} \leq g_i < z^{k_i + 1}$.   We claim that this finite collection of inequalities also implies that $\langle z \rangle$ is cofinal.  Indeed, that $z$ is central implies that 
$z^{k_i + k_j} \leq g_i g_j < z^{k_i + k_j + 2}$
for all $i, j$, and inductively, that each word in the generators is bounded above and below by powers of $z$.  
\end{proof}

\begin{remark} 
This can fail when $G$ is not finitely generated.  Indeed, let $G=\Z[x]$ under addition, and let $z=1\in \Z[x]$. We can order $\Z[x]$ by letting $0 <_\pi p(x)$ if and only if $ p(\pi)$ is a non-negative real number. On the other hand any finite set $S$ of $G$ lies inside a subgroup $H$ that admits a complement, say $H\oplus H'=G$. On $H$ we can put again $<_\pi$, and extend this ordering lexicographically  to $G$ using any ordering on $H'\simeq G$, so that $H$ is a proper convex subgroup. The resulting ordering $<'$ on $G$ agrees with $<_\pi$ on $S$, so by choosing $S$ large it can be made arbitrarily close to $<_\pi$. However, if on $H'$ we put the lexicographic ordering coming from the natural identification of $\Z[x]$ with a direct sum of infinitely many copies of $\Z$, then every element of $G$ fixes a point in the dynamical realization of $<'$.
\end{remark}

Using Lemma \ref{open lem} and Proposition \ref{cofinal prop}, we can produce isolated left orders from isolated circular orders on finitely generated groups.

\begin{proposition}   \label{LO prop}
Assume that $G$ is finitely generated and $c$ is an isolated circular order on $G$.  If $\Z \to \hat{G} \overset{\pi}\to G$ is a central extension and $< \in \LO_\Z(G)$ a left order such that $\pi^\ast(<) = c$, then $<$ is isolated in $\LO(\hat{G})$.  
\end{proposition}

\begin{proof} 
By Proposition \ref{cofinal prop}, $\pi^\ast: \LO(\hat{G}) \to \CO(G)$ is continuous.  We claim that, since $G$ (and hence $\hat{G}$) is finitely generated, $\pi^\ast$ is also locally injective.  To see this, suppose $<$ and $<' \in \LO(\hat{G})$ have the same image under $\CO(G)$.  Let $\rho$ and $\rho'$ be the dynamical realizations of $<$ and $<'$ respectively; up to conjugacy, we can assume that $\rho(z) = \rho'(z) = x\mapsto x+1$, where $z$ is the generator of the $\Z$ subgroup of $\hat{G}$. We can also assume that the induced actions of $G$ on $\R/\Z$ agree, since this action is the dynamical realization of $\pi^\ast(<) = \pi^\ast(<')$.  Fix a finite generating set $S = \{s_1, ... s_n\}$ for $G$, and choose generators $\{\hat{s}_1, ... \hat{s}_n, z\}$ for $\hat{G}$ where $\hat{s}_i$ is the (unique) lift of $s_i$ to $\hat{G}$ such that $\id < \hat{s}_i < z$.   Since $\pi^\ast(<) = \pi^\ast(<')$, for each $s \in S$ we have $\rho(\hat{s}_i) = \rho'(s_i) z^{n_i}$ for some $n_i \in \Z$.  In particular, if $<$ and $<'$ are close enough so that they agree on the elements $\hat{s}_i$ and $z$, we will necessarily have $\rho(\hat{s}_i) = \rho'(\hat s_i)$, thus the dynamical realizations, and hence the orders $<$ and $<'$ agree.  

Now let $<$ be an order with $\pi^\ast(<)$ isolated, as in the statement of the proposition.  Continuity and local injectivity of $\pi^\ast$ combined with the fact that $\pi^{\ast}(<)$ is isolated implies that $<$ must have also been isolated in $\LO_\Z(\hat{G})$.  Lemma \ref{open lem} says that $\LO_\Z(\hat{G})$ is an open neighborhood of $<$ in $\LO(\hat{G})$, so $<$ is isolated in $\LO(\hat{G})$.  
\end{proof}

Proposition \ref{cofinal prop}, \ref{LO prop}, and Theorem \ref{f_2_iso_thm} together imply the following.  

\begin{corollary}
The space of left orders of $F_2 \times \Z $  has infinitely many nonconjugate isolated points. 
\end{corollary}

\begin{proof}
The existence of infinitely many isolated points is an immediate consequence of \ref{cofinal prop}, \ref{LO prop} and Theorem \ref{f_2_iso_thm}, together with the fact that all central extensions of $F_2$ by $\Z$ are trivial, i.e. direct products.   The argument that the examples obtained by pulling back the nonconjugate ``standard $k$-lift" orders on $F_2$ to $F_2 \times \Z$ are nonconjugate can be done similarly to the argument that the original orders on $F_2$ were non-conjugate.  Since these dynamical realizations have image in $\Homeo_\Z(\R)$, elements have a well defined, conjugation-invariant, \emph{translation number}, and the translation number of the commutator $[a,b]$ in the $k$-fold lift is $1/k$ (again, this is standard and more background can be found in \cite{Mann invent}).  Thus, varying $k$ gives non-conjugate dynamical realizations, and hence nonconjugate orders.  
\end{proof}

This example is interesting for three reasons. First, $F_2\times \Z$ is isomorphic to $P_3$, the pure braid group on three strands.  It is known that  the braid group $B_n$ admits isolated orderings. Could they always come from isolated ordering on $P_n$?  Related work in an extensive study of orderings on braid groups can be found in \cite{braids}. Secondly, $F_2\times\Z$ is an example of a right-angled Artin group.  Perhaps it is possible to give a complete characterization of which RAAGs have isolated left orders.  Finally, this example also shows that direct products behave quite differently than free products --  in \cite{Rivas 12}, it is shown that the free product of any two left orderable groups has no isolated left orders.

\section{Further questions}  \label{questions sec}

Since we have found infinitely many isolated points in the compact space $\CO(F_{2n})$, they must accumulate somewhere.  It is not hard to show the following, we leave the proof as an exercise. 

\begin{proposition}[Accumulation points of lifts in $\CO(F)$]
Let $A \subset \CO(F_2)$ denote the set of all finite lifts of the Fuchsian circular order.  Then the accumulation points of $A$ are left orders on $F_2$.   These left orders have dynamical realizations conjugate into the universal covering group of $\PSL(2,\R)$, which acts by homeomorphisms on the line commuting with integer translations.  
\end{proposition} 

Since $\LO(F_2)$  has no isolated points, the accumulation points of $A$ belong to a Cantor set embedded in $\CO(G)$.  Remarkably, we do not know a single example of any group $G$ such that $\CO(G)$ -- or even $\LO(G)$ -- has accumulation points that do not belong to a Cantor set.  

\begin{question} 
Does there exist a countable group $G$ such that the derived set of $\CO(G)$ is neither empty nor a Cantor set?  
\end{question}

Another question that remains open is the following:

\begin{question}  \label{odd q}
Does there exist an isolated circular order on $F_n$, for $n$ odd?  
\end{question}

Theorem \ref{domains thm} reduces this to an essentially combinatorial problem of checking whether an arrangement of attracting domains gives an action which permutes transitively the complementary intervals to the minimal set.   At this point we do not know a single example of an isolated circular order, and suspect the answer to Quesiton \ref{odd q} may well be negative.


\vspace{.3in}

\textit{Kathryn Mann}

Dept. of Mathematics, University of California, Berkeley  

kpmann@math.berkeley.edu

\bigskip 

\textit{Crist\'obal Rivas}

Dpto. de Matem\'aticas y C.C., Universidad de Santiago de Chile

cristobal.rivas@usach.cl


\begin{thebibliography}{99}

\bibitem{BS} H. Baik, E. Samperton,
\textit{Spaces of invariant circular orders of groups}.
Preprint. 	arXiv:1508.02661 [math.GR] (2015).   To appear in Groups, Geometry and Dynamics.  

\bibitem{Calegari} D. Calegari, 
\textit{Circular groups, planar groups, and the Euler class}. 
Geometry and Topology Monographs, Proceedings of the Casson Fest 7 (2004), 431-491.

\bibitem{CD}
D. Calegari, N. Dunfield, 
\textit{Laminations and groups of homeomorphisms of the circle}. 
Invent. Math. 152, no. 1 (2003) 149-204.

\bibitem{Dalbo} F. Dal'bo
\textit{Geodesic and horocyclic trajectories}.
Springer Universitext, Springer-Verlag, London, 2011.

\bibitem{braids} 
P. Dehornoy, I. Dynnikov, D. Rolfsen, B. Weist,
\textit{Ordering Braids.}
Mathematical Survey and Monographs 148, 2008.

\bibitem{GOD} 
B. Deroin, A. Navas, C. Rivas,
\textit{Groups, Orders, and Dynamics}.
Preprint.  arXiv:1408.5805 [math.GR] (2015).

\bibitem{Ghys Ens} E. Ghys,
\textit{Groups acting on the circle}.
L'Ens. Math\'ematique, 47 (2001), 329-407. 

\bibitem{Jakubik}
J. Jakub\'ik, G. Pringerov\'a. 
\textit{Representations of cyclically ordered groups.} 
Casopis pro pestov\'an\'i matematiky 113.2 (1988), 184-196.

\bibitem{Mann invent} K. Mann, 
\textit{Components of spaces of surface group representations}. 
Invent. Math. 201, no. 2 (2015), 669-710.

\bibitem{Mann survey} K. Mann,
\textit{Rigidity and flexibility of group actions on $S^1$}.
To appear in \textit{Handbook of Group Actions}. L. Ji, A. Papadopoulos, and S.-T. Yau, eds. International Press, Boston.

\bibitem{Margulis} G. Margulis, 
\textit{Free subgroups of the homeomorphism group of the circle}. 
C. R. Acad. Sci. Paris, Ser. I Math. 331 no. 9, (2000), 669-674.

\bibitem{Matsumoto invent} S. Matsumoto,
\textit{Some remarks on foliated $S^1$ bundles}.
Invent. Math. 90 (1987), 343-358.

\bibitem{Matsumoto new} S. Matsumoto, 
\textit{Basic partitions and combinations of group actions on the circle: A new approach to a theorem of Kathryn Mann}.
Preprint. arXiv:1412.0397 [math.GT] (2014).

\bibitem{McCleary} S.H. McCleary, 
\textit{Free lattice ordered groups represented as $o$-2 transitive $\ell$-permutation groups}. 
trans. Amer. Math. Soc. 209(2) (1985), 69-79.


\bibitem{Navas orders} A. Navas,
\textit{On the dynamics of left-orderable groups}.
Ann. Inst. Fourier (Grenoble) 60 (2010), 1685-1740.

\bibitem{Navas 14} A. Navas, 
\textit{Sur les rapprochements par conjugaison en dimension 1 et classe $C^1$}.
Compositio Mathematica 150.7 (2014), 1183-1195.

\bibitem{Sikora} A. Sikora,
\textit{Topology on the spaces of orderings of groups.}
Bull. Lon. Math. Soc. 36.4 (2004), 519-526.

\bibitem{Rivas jgt} 
C. Rivas,
\textit{On the space of Conradian group orderings.}
J. of Group Theory, 13 (2010) 337-353.  

\bibitem{Rivas 12} 
C. Rivas,
\textit{Left-orderings on free products of groups}.
Journal of Algebra, 350 (2012) 318-329.  

\bibitem{Rivas-Tessera} 
C. Rivas, R. Tessera,
\textit{On the space of left-orderings of virtually solvable groups.}
Groups, Geometry, and Dynamics, 10 (2016) 65-90.  

\bibitem{Thurston}
W. Thurston
\textit{3--manifolds, foliations and circles II}.
Unpublished preprint.  

\bibitem{Zheleva} 
S.D. Zheleva,
\textit{Cyclically ordered groups.}
Siberian Math. Journal. 17 (1976), 773-777.


\end{thebibliography}
\end{document}